\documentclass[3p,times]{elsarticle}
\makeatletter
\def\ps@pprintTitle{   \let\@oddhead\@empty
    \let\@evenhead\@empty
     \def\@oddfoot{\centerline{\thepage}}      \let\@evenfoot\@oddfoot}
      \makeatother

\usepackage{amssymb}
\usepackage{amsthm}
\usepackage{mathtools}
\usepackage{pgfplots}
\usepackage{booktabs}
\usepackage{hyperref}
\usepackage{siunitx}

\definecolor{red}{rgb}{0.7592, 0.3137, 0.3020}
\definecolor{blu}{rgb}{0.3098, 0.5059, 0.7412}
\definecolor{grn}{rgb}{0,0.4,0}
\definecolor{org}{rgb}{1.0,0.5,0.0}
\definecolor{prp}{rgb}{0.4,0.35,0.8}

\newcommand{\eref}[1]{{(\ref{#1})}}
\newcommand{\mat}[1]{{\boldsymbol{#1}} }
\renewcommand{\vec}[1]{{\boldsymbol{#1}} }
\newcommand{\LL}{\ensuremath{l} }
\newcommand{\avg}[1]{\ensuremath{\left\{\!\left\{ #1 \right\}\!\right\} } }
\newcommand{\jmp}[1]{\ensuremath{\left[\!\left[ #1 \right]\!\right]}}

\newcommand{\EE}{\ensuremath{\mathcal{E}} }
\newcommand{\GG}{\ensuremath{\mathcal{G}} }
\newcommand{\En}{\ensuremath{\mathcal{F}} }\newcommand{\GD}{\ensuremath{W} }\newtheorem{theorem}{Theorem}[section]
\newtheorem{corollary}{Corollary}[section]

\renewcommand{\underbar}[1]{\mkern1.5mu\underline{\mkern-1.5mu#1\mkern-1.5mu}\mkern1.5mu}

\usepackage[figuresright]{rotating}

\begin{document}

\begin{frontmatter}

\title{Robust Approaches to Handling Complex Geometries with Galerkin Difference
  Methods}

  \author[NPS]{Jeremy E. Kozdon\corref{cor1}}
  \ead{jekozdon@nps.edu}
  \cortext[cor1]{corresponding author}
  \author[NPS]{Lucas C. Wilcox}
  \ead{lwilcox@nps.edu}
  \author[SMU]{Thomas Hagstrom}
  \ead{thagstrom@smu.edu}
  \author[RPI]{Jeffrey W. Banks}
  \ead{banksj3@rpi.edu}

  \address[NPS]{Department of Applied Mathematics,
          Naval Postgraduate School,
          833 Dyer Road,
          Monterey, CA 93943--5216}
  \address[SMU]{Department of Mathematics,
          Southern Methodist University,
          PO Box 750156,
          Dallas, TX 75275--0156}
  \address[RPI]{Department of Mathematical Sciences,
          Rensselaer Polytechnic Institute,
          110 8th Street,
          Troy, New York 12180.}
  \nonumnote{The views expressed in this document are those of the
  authors and do not reflect the official policy or position of the Department
  of Defense or the U.S.  Government.\\ Approved for public release;
  distribution unlimited\\}

\begin{abstract}
  The Galerkin difference (GD) basis is a set of continuous, piecewise polynomials
  defined using a finite difference like grid of degrees of freedom. The one
  dimensional GD basis functions are naturally extended to multiple dimensions
  using the tensor product constructions to quadrilateral elements for
  discretizing partial differential equations. Here we
  propose two approaches to handling complex geometries using the GD basis
  within a discontinuous Galerkin finite element setting: (1) using
  non-conforming, curvilinear GD elements and (2) coupling affine GD elements
  with curvilinear simplicial elements. In both cases the (semidiscrete)
  discontinuous Galerkin method is provably energy stable even when variational
  crimes are committed and in both cases a weight-adjusted mass matrix is used,
  which ensures that only the reference mass matrix must be inverted.
  Additionally, we give sufficient conditions on the treatment of metric terms
  for the curvilinear, nonconforming GD elements to ensure that the scheme is
  both constant preserving and conservative. Numerical experiments confirm the
  stability results and demonstrate the accuracy of the coupled schemes.
\end{abstract}


\end{frontmatter}

\section{Introduction}
The aim of this work is to propose two possible approaches to handling complex
geometries in a robust and efficient manner when Galerkin difference (GD)
methods are used to approximate partial differential equations.
GD methods, initially proposed in~\cite{BanksHagstrom2016JCP},
are element based methods where the approximation space over each element is
based on a finite difference like grid of degrees of freedom. Between the grid
points (called subcells below) a polynomial is built using neighboring grid
values; the support of the polynomial over a subcell is finite, but the support
extends beyond the subcell. It is worth noting that in general one would have
many GD subcells in each GD element and the number of degrees of freedom inside
an GD element is independent of the polynomial order. The GD
method may be extended to multiple dimensions using a tensor product
construction~\cite{BanksHagstromJacangelo2018JCP}; details of the GD
approximation space are discussed in Section~\ref{sec:GDbasis}.

Since the GD method is an element based method, two-dimensional complex
geometries can be handled by partitioning the domain into a set of curvilinear
quadrilateral elements. Since the GD method is based on a finite difference like
grid of degrees of freedom it is desirable to have large elements with
resolution requirements handled by increasing the number of degrees of freedom
within the element. This may pose a challenge for two reasons: (1) if the
elements are required to be conforming (e.g., the same number of grid points are
required on both sides of an element interface) unnecessarily fine grids maybe
be required in some regions and (2) some of the cells maybe be so small that just
having a few GD subcells inside these elements may require extremely fine grids.

To address these challenges, we propose two approaches: (1) the use of
nonconforming, curvilinear GD elements and (2) coupling affine GD elements with
curvilinear, polynomial simplicial elements. An example of both approaches is
shown in Figure~\ref{fig:decomp}. In both cases, the method used is a
discontinuous Galerkin method where the interelement coupling (whether GD-GD,
GD-simplicial, or simplicial-simplicial) is done using numerical fluxes. The
advantage of using a discontinuous Galerkin method as the underlying coupling
methodology is that the resulting global mass matrix is block diagonal,
which only requires the efficient inversion of elemental mass matrices.

A desirable property of any numerical method is stability. When exact
integration is used, stability of the coupled method follows directly from the
variational form. That said in practice, when the elements are curved, exact
integration is rarely used and variational crimes can cause instability (e.g.,
positive real parts to the semidiscrete method's eigenvalue spectrum). In recent
years skew-symmetry has been found to robustify high-order methods when the
quadrature is inexact;
see~\cite{ChanWangModaveRemacleWarburton2016jcp,FisherCarpenterNordstromYamaleevSwanson2012jcp,Gassner2013,KoprivaGassner2014SISC,KozdonDunhamNordstrom2013,Nordstrom2006,Warburton2013}.
The advantage of the skew-symmetric form is that the divergence
theorem (i.e., integration-by-parts) need not hold discretely for the stability analysis to
follow, and the volume and surface stability are decoupled. The skew-symmetric
form has also been shown to be useful even on affine elements when the mesh is
nonconforming~\cite{FriedrichEtAl2017,KozdonWilcox2018}. The particular
skew-symmetric formulation we used is based on~\cite{KozdonWilcox2018}; see
Section~\ref{sec:form}. An important property of the formulation is that it is
provably stable even with inexact quadrature; see Section~\ref{sec:semi:disc}.

The compact support of the GD method basis functions means that the mass matrix
will be banded, and thus when used in a finite element framework the inverse can
be applied efficiently through a banded Cholesky factorization.  Furthermore,
when the element Jacobian determinant is constant (e.g., the element is affine)
the mass matrix has a tensor product structure allowing a dimension-by-dimension
application of the inverse along the grid lines.  That said, when curved
elements (or mapped geometries) are used the tensor product structure is lost
and, in the na\"{\i}ve implementation, the full mass matrix must be factored;
since the basis functions have compact support the mass matrix remains banded
but the size and bandwidth are increased. Having to factor the curved mass
matrix increases the initialization cost and time-stepping
of the method as well as the storage
requirements. To address this, we propose using a weight-adjusted mass
matrix~\cite{ChanHewettWarburton2017_curviWADG}. In this approach the inverse of
the Jacobian weighted mass matrix is approximated in a manner that only requires
the inverse of the reference mass matrix and the application of the mass matrix
weighted with the reciprocal of the Jacobian determinant. We do this both for
the GD elements and the curved simplicial elements, and when applied in the GD
method the tensor product structure can be exploited. Details for the
weight-adjusted approach are given Section~\ref{sec:WADG}.

Another important computational question is the storage of the metric terms. As
will be seen when the metric terms are non-constant a set of GD quadrature nodes
must be used. This grid of quadrature nodes is fine compared with the GD grid of
degrees of freedom since many quadrature nodes are needed over each GD subcell.
If the metric terms are stored at the quadrature nodes then the storage
requirements are drastically increased, thus we propose computing the
$L^{2}$-projection of the metric terms to the GD approximation space.
Furthermore, we also show how the metric terms can be computed across
non-conforming curved interfaces so that the scheme is both conservative and
constant preserving; see Section~\ref{sec:GD:metric}.

In the remainder of the text, we will use linear acoustics as a model problem
with constant material properties. The approach can be extended to any linear
equation with has a skew-symmetric formulation.

\begin{figure}
  \centering
  \includegraphics{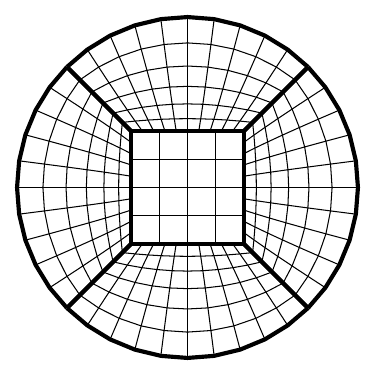}
  \includegraphics{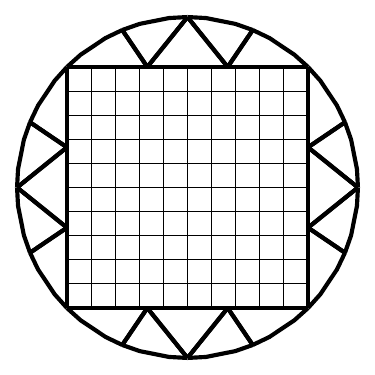}
  \caption{Example domain decompositions for two approaches to handling complex
    geometries with GD methods.
    The thick lines show the boundaries of the elements and the thin lines show the GD
    grid lines (i.e., subcell boundaries).
    (left) A decomposition with curved GD elements.
    (right) A decomposition with an affine GD element and several
    curved simplicial elements.\label{fig:decomp}}
\end{figure}

\section{Model Problem and Skew Symmetric Discontinuous Galerkin
Scheme}\label{sec:form}
For our model problem, we consider the two-dimensional, constant coefficient
acoustic wave equation in first order form:
\begin{align}
  \label{eqn:strong}
  \frac{\partial p}{\partial t} + \nabla\cdot\vec{v} = 0,
  \quad
  \frac{\partial \vec{v}}{\partial t} + \nabla p = \vec{0},
  \quad\mbox{in}\quad
  \Omega,
\end{align}
where $p$ is the pressure and $\vec{v} = {[v_{x}~v_{y}]}^{T}$ is the
particle velocity in the $x$ and $y$ directions.  Here $\Omega \subset
\mathbb{R}^{2}$ is the domain of the problem. For simplicity, only periodic and
the zero normal velocity boundary condition,
\begin{align}
  \label{eqn:bc}
  \vec{n} \cdot \vec{v} = 0
  \quad\mbox{on}\quad
  \partial \Omega,
\end{align}
are considered. In~\eref{eqn:bc} $\vec{n}$ is the outward pointing unit normal
to $\partial \Omega$.

The discontinuous Galerkin (DG) discretization of the acoustic wave
equation begins with a partition of the domain into a set, $\EE$, of
non-overlapping elements such that $\Omega = \bigcup_{e\in\EE} e$.
Following~\cite{KozdonWilcox2018}, we allow the elements to be non-conforming by
patching them together with a mortar, $\Gamma = \bigcup_{e\in\EE} \partial e$,
where $\partial e$ is the boundary of element $e$. The mortar is partitioned
into a set, $\GG$, of non-overlapping, one-dimensional mortar elements such that
$\Gamma = \bigcup_{g\in\GG} g$ and each mortar element intersects at most the
boundary of two elements.
The set of mortar elements that a volume element $e \in \EE$ connects to is
\begin{align}
  \GG^{e} =
  \left\{\,g \in \GG \mid g \cap \partial e \ne \emptyset \,\right\}.
\end{align}

For each element, $e\in\EE$, we assume there exists a diffeomorphic mapping
from the reference element $\hat{e}$ to $e$. We let $\{x^{e}(r,s), y^{e}(r,s)\}$
be this mapping and $\{r^{e}(x,y), s^{e}(x,y)\}$ be its inverse.
Likewise, for each mortar element, $g\in\GG$, we assume there exists a
diffeomorphic mapping from the reference element $\hat{g}$ to $g$. We let
$\{x^{g}(r), y^{g}(r)\}$ be this mapping and $\{r^{g}(x,y)\}$ be its inverse.

Let $V^{e}_{h}$ be a finite-dimensional approximation space over $\hat{e}$.
Below, depending on the type of element, this is the span of a
Galerkin difference (GD) basis, described in Section~\ref{sec:GDbasis}, or the span
of polynomials up to a given order.

The division of the mortar and the approximation space for each reference mortar
element are chosen to represent the trace of functions from
$V^{e}_{h}$ exactly.
Depicted in Figure~\ref{fig:mortar} is how the mortar space would be partitioned
when two GD elements are connected (left panel) and a single GD
element connects to several simplicial elements (right panel).
The thick lines denote interfaces between elements and the thin lines
denote subcells of the larger GD element that support a single tensor product
polynomial in the reference space (see Section~\ref{sec:GDbasis} for more detail on
the GD basis).
Each side of a mortar element connects to either a single GD subcell or
simplicial element.
The approximation space for the reference mortar element, $\hat{g}$, is
the space of one-dimensional polynomials of degree less than or equal to $n$,
denoted $\mathbb{P}_{n}$, with $n$ being the maximum polynomial order on the
two-sides of the mortar.

A skew-symmetric discontinuous Galerkin semi-discretization of \eref{eqn:strong}
is: For each element $e\in\EE$ find a $p \in [0, T] \times V_{h}^{e}$ and
$\vec{v} \in {\left([0, T] \times V^{e}_{h}\right)}^{2}$ such that
\begin{align}
  \label{eqn:var:p}
  \int_{\hat{e}}
  \left( J \phi \frac{\partial p}{\partial t}
  + J\phi \nabla\cdot\vec{v}\right)
  &= -\sum_{g\in\GG^{e}} \int_{\partial\hat{g}} S_{J}^{g} \phi^{-}
  \left(v_{n}^{*} - v_{n}^{-}\right),\\
  \label{eqn:var:v}
  \int_{\hat{e}}
  \left(J {\vec{\omega}}^{T} \frac{\partial \vec{v}}{\partial t}
  -  J \left(\nabla\cdot\vec{\omega}\right) p\right)
  &= -\sum_{g\in\GG^{e}} \int_{\partial\hat{g}} S_{J}^{g} \omega^{-}_{n} p^{*},
\end{align}
for all $\phi \in V^{e}_{h}$ and $\vec{\omega} \in
{\left(V^{e}_{h}\right)}^{2}$.  Here $\phi^{-}$ denotes the trace of the volume
element field $\phi$ on the mortar element $g$. Likewise $v_{n}^{-} =
\vec{n}^{-}\cdot\vec{v}^{-}$ and $\omega_{n}^{-} =
\vec{n}^{-}\cdot\vec{\omega}^{-}$ where $\vec{n}^{-}$ is
the outward pointing normal of element $e$. The volume and surface Jacobian
determinants are
$J$ and $S_{J}^{g}$, respectively.  The terms $p^{*}$ and $v_{n}^{*}$ are the
numerical fluxes, which couple the solution across the element interfaces and
are defined as
\begin{align}
  \label{eqn:flux}
  p^{*} &= \avg{p} - \frac{\alpha}{2}\jmp{v_{n}},&
  v_{n}^{*} &= \avg{v_{n}} - \frac{\alpha}{2}\jmp{p},
\end{align}
for some constant $\alpha \ge 0$.  Here $\avg{p} = (p^{+} + p^{-}) / 2$ is the
average and $\jmp{p} = p^{+} - p^{-}$ is the jump across the mortar with $p^{-}$
denoting the trace of $p$ from $e$ on $g$ and $p^{+}$ the trace
from the other element connected on $g$. At the outer boundaries, the zero
velocity boundary condition is enforced by setting $p^{+} = p^{-}$ and
$v_{n}^{+} = - v_{n}^{-}$.

\begin{figure}
  \centering
  \includegraphics{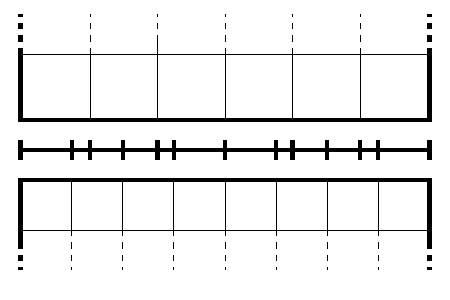}
  \includegraphics{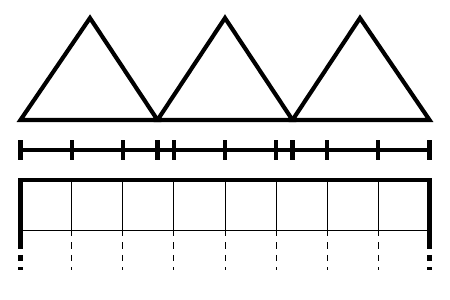}
  \caption{    Mortar elements that would be used to connect two GD elements (left)
    and a GD element with several simplicial elements (right).
    The elements have been separated vertically to show the mortar elements,
    which are represented by the hash marks on the line between the elements.
    In both panels the dashed line indicates that the GD element has been
    truncated and the thin lines show the GD grid lines
    (i.e., subcell boundaries).
    In the right panel only elements that connect to the GD element are shown
    and mortar elements between simplicial elements are not shown.\label{fig:mortar}}
\end{figure}

\section{Semi-Discrete Scheme}\label{sec:semi:disc}
In this section we define a quadrature-based discretization of
\eref{eqn:var:p}--\eref{eqn:var:v} as well as introduce the matrix vector form
of the problem \eref{eqn:semi:p}--\eref{eqn:semi:vy}. To start with we leave the
notation fairly general so that it includes both the (yet to be defined) GD
elements and simplicial elements; details of the specific
choices for each of these operators is given later in the text for the GD basis.
We refer the reader to Hesthaven and Warburton~\cite{HesthavenWarburton2008} for
a detailed description of the polynomial basis for simplicial elements. In what follows
we describe the operators from the perspective of a single element. These
operators can vary from element to element depending on the element type and
metric terms.  That said, in order to simplify the notation we suppress any
superscripts or subscripts to denote element unless this is needed for clarity,
for instance when discussing mortar elements.

Let ${\{\phi_{k}\}}_{k=0}^{N} \in V^{e}_{h}$ be a basis for the reference space
of element $e$. Thus, any function $q \in V^{e}_{h}$ can be represented as
\begin{align}
  q = \sum_{k=0}^{N} q_{k} \phi_{k} = \vec{\phi}^{T} \vec{q},
\end{align}
with $\vec{q} = {[q_{0}\;\dots\;q_{N}]}^{T}$ and $\vec{\phi} =
{[\phi_{0}\;\dots\;\phi_{N}]}^{T}$. We approximate integrals over reference
element $\hat{e}$ with the $N_{q}$-point, positive weight, interpolatory
quadrature rule
\begin{align}
  \int_{\hat{e}} f \approx \sum_{k=1}^{N_q} \omega_{k} f(r_{k},s_{k}),
\end{align}
where the quadrature weights ${\{\omega_{k}\}}_{k=1}^{N_{q}}$ are strictly
positive and quadrature node locations ${\{(r_{k},s_{k})\}}_{k=1}^{N_{q}}$ are
given for the reference element. Inner products of $p,q \in V^{e}_{h}$ weighted
by $J \in
L^{2}(\hat{e})$ can then be approximated as
\begin{align}
  \label{eqn:approx:inner}
  \int_{\hat{e}} J pq \approx \sum_{k=1}^{N_q} \omega_{k} J(r_{k},s_{k})
  p(r_{k},s_{k}) q(r_{k},s_{k}) = \vec{p}^{T} \mat{L}^{T} \mat{W}\mat{J}
  \mat{L}\vec{q},
\end{align}
where $\mat{W}$ is a diagonal matrix of the quadrature weights,
$\mat{J}$ is a diagonal matrix of $J$ evaluated at the quadrature nodes, and
$\mat{L}$ interpolates functions in $V^{e}_{h}$ which have been expanded in
basis $\vec{\phi}$ to the integration nodes.  For convenience, we define the
$J$-weighted mass matrix to be
\begin{align}
  \mat{M}_{J} = \mat{L}^{T} \mat{W}\mat{J} \mat{L}.
\end{align}
Since the quadrature weights are positive and the nodes are distinct, the mass
matrix $\mat{M}_{J}$ is symmetric positive definite as long as $J > 0$ at the
quadrature nodes and $N \le N_{q}$, i.e., there are more quadrature nodes than
basis functions for the space $V^{e}_{h}$.

We let $\mat{D}_{r}$ and $\mat{D}_{s}$ be operators that differentiate functions
in $V^{e}_{h}$ that have been expanded in basis $\vec{\phi}$ in the reference
directions $r$ and $s$, respectively. These operators evaluate the derivatives
at the quadrature nodes. Letting $p,v_{x}, v_{y} \in V^{e}_{h}$ with $\vec{v} =
{[v_{x}\;v_{y}]}^{T}$ we then have the following approximation:
\begin{align}
  \int_{\hat{e}} J p \nabla\cdot\vec{v}
  &=
  \int_{\hat{e}} J p \left(\frac{\partial v_{x}}{\partial x} + \frac{\partial
  v_{y}}{\partial y}\right)
  =
  \int_{\hat{e}} J p
  \left(
  \frac{\partial r}{\partial x} \frac{\partial v_{x}}{\partial r}
  +
  \frac{\partial s}{\partial x} \frac{\partial v_{x}}{\partial s}
  +
  \frac{\partial r}{\partial y} \frac{\partial v_{y}}{\partial r}
  +
  \frac{\partial s}{\partial y} \frac{\partial v_{y}}{\partial s}
  \right)
  \notag\\
  \label{eqn:int:pgradv}
  \approx&
  \vec{p}^{T} \mat{L}^{T} \mat{W} \mat{J}
  \left(
  \mat{r}_{x} \mat{D}_{r} \vec{v}_{x}
  +
  \mat{s}_{x} \mat{D}_{s} \vec{v}_{x}
  +
  \mat{r}_{y} \mat{D}_{r} \vec{v}_{y}
  +
  \mat{s}_{y} \mat{D}_{s} \vec{v}_{y}
  \right).
\end{align}
Here $\mat{r}_{x}$ is a diagonal matrix of $\partial r/\partial x$ evaluated at
the quadrature nodes, with similar definitions for $\mat{s}_{x}$, $\mat{r}_{y}$,
and $\mat{s}_{y}$. It is convenient to define the stiffness matrices
\begin{align}
  \mat{S}_{x} =&
  \mat{L}^{T} \mat{W} \mat{J}
  \left(
  \mat{r}_{x} \mat{D}_{r}
  +
  \mat{s}_{x} \mat{D}_{s}
  \right),&
  \mat{S}_{y} =&
  \mat{L}^{T} \mat{W} \mat{J}
  \left(
  \mat{r}_{y} \mat{D}_{r}
  +
  \mat{s}_{y} \mat{D}_{s}
  \right),
\end{align}
so that we can write \eref{eqn:int:pgradv} as
\begin{align}
  \int_{\hat{e}} J p \nabla\cdot\vec{v}
  \approx
  \vec{p}^{T} \left(\mat{S}_{x} \vec{v}_{x} + \mat{S}_{y} \vec{v}_{y}\right).
\end{align}

Integrals over mortar element $g \in \GG^{e}$ are approximated using the
$N_{q}^{g}$-point, positive weight, interpolatory quadrature rule
\begin{align}
  \int_{\hat{g}} f \approx \sum_{k=1}^{N_q^{g}} \omega_{k}^{g} f(r_{k}^{g}),
\end{align}
where ${\{\omega_{k}^{g}\}}_{k=1}^{N_{q}^{g}}$ and
${\{r_{k}^{g}\}}_{k=1}^{N_{q}^{g}}$ are the surface quadrature weights and
nodes, respectively. The $S_{J}^{g}$-weighted inner product of functions
$\bar{p},\bar{q} \in \mathbb{P}_{N}$ over $\hat{g}$ is approximated by
\begin{align}
  \int_{\hat{g}} S_{J}^{g} \bar{p} \bar{q}
  \approx
  \sum_{k=1}^{N_{q}^{g}} \omega_{k}^{g} S_{J}^{g}(r_{k}^{g})
  \bar{p}(r_{k}^{g}) \bar{q}(r_{k}^{g})
  =
  \vec{\bar{p}}^{T} \mat{S}_{J}^{g} \mat{W}^{g} \vec{\bar{q}},
\end{align}
where $\mat{W}^{g}$ is a diagonal matrix of the quadrature weights,
$\mat{S}_{J}^{g}$ is a diagonal matrix of $S_{J}^{g}$ evaluated at the
quadrature nodes, and vectors $\vec{\bar{p}}$ and $\vec{\bar{q}}$ are $\bar{p}$
and $\bar{q}$ evaluated at the quadrature nodes.
Evaluation of mortar integrals involving the trace of $p, q \in
V^{e}_{h}$ on the mortar (e.g., $p^{-}$ and $q^{-}$) can be evaluated by
defining the matrix $\mat{L}^{g}$ as the interpolation matrix from $\hat{e}$ to
the quadrature nodes over $\hat{g}$. With this we then have the integral
approximation
\begin{align}
  \int_{\hat{g}} S_{J}^{g} p^{-} q^{-}
  \approx
  \vec{p}^{T} {(\mat{L}^{g})}^{T} \mat{S}_{J}^{g} \mat{W}^{g} \mat{L}^{g} \vec{q}.
\end{align}
For stability, the same surface Jacobian determinant must be used for the volume
update on both sides of the mortar. Since the mortar element space need not be a
subspace of the trace of the volume element space, the representation of the
surface Jacobian determinant on mortar element $g$ may be different than the
representation of surface Jacobian determinant on the connected face of element
$e$. For GD elements this is important, since the only functions that can be
represented exactly on both side of a nonconforming GD interface are polynomials
(as opposed to piecewise polynomials); see~\cite{KozdonWilcox2018} for a fuller
discussion of the mortar elements.

With this notation, a skew-symmetric quadrature based version of
\eref{eqn:var:p}--\eref{eqn:var:v} is then: For each $e \in \EE$ find a
$p,\;v_{x},\;v_{y} \in V^{e}_{h} {\left([0, T] \times V^{e}_{h}\right)}$ such
that
\begin{align}
  \label{eqn:semi:p}
  &\mat{M}_{J} \frac{d\vec{p}}{dt}
  + \mat{S}_{x} \vec{v}_{x} + \mat{S}_{y} \vec{v}_{y}
  = -\sum_{g\in\GG^{e}} {(\mat{L}^{g})}^{T} \mat{S}_{J}^{g} \mat{W}^{g}
  \left(\vec{v}^{*}_{n} - \vec{v}^{-}_{n}\right),\\
  \label{eqn:semi:vx}
  &\mat{M}_{J} \frac{d\vec{v_{x} }}{dt}
  - \mat{S}_{x}^{T} \vec{p}
  = -\sum_{g\in\GG^{e}} {(\mat{L}^{g})}^{T} \mat{n}_{x}^{-g} \mat{S}_{J}^{g} \mat{W}^{g}
  \vec{p}^{*},\\
  \label{eqn:semi:vy}
  &\mat{M}_{J} \frac{d\vec{v_{y} }}{dt}
  - \mat{S}_{y}^{T} \vec{p}
  = -\sum_{g\in\GG^{e}} {(\mat{L}^{g})}^{T} \mat{n}_{y}^{-g} \mat{S}_{J}^{g} \mat{W}^{g}
  \vec{p}^{*}.
\end{align}
Here $\mat{n}_{x}^{-g}$ and $\mat{n}_{y}^{-g}$ are diagonal matrices of the
components of the unit normal to mortar $g$ evaluated at the mortar element
quadrature points; the normal is defined to be outward with respect to element
$e$.  Note, it is important for stability that the normal used for the element
on the opposite side of the mortar be equal in magnitude but opposite in sign.
The normal velocity on the mortar is defined to be
\begin{align}
   \vec{v}_{n}^{-} =
   \mat{n}_{x}^{-g} \mat{L}^{g} \vec{v}_{x}
   +
   \mat{n}_{y}^{-g} \mat{L}^{g} \vec{v}_{y}.
\end{align}
The flux vectors $\vec{p}^{*}$ and $\vec{v}^{*}_{n}$ are defined by applying
\eref{eqn:flux} pointwise.

We again note that scheme \eref{eqn:semi:p}--\eref{eqn:semi:vy} is written from
the point of view of a single element $e$, and each element $e$ will have unique
solution vectors and matrices (e.g., in general the mass matrix and stiffness
matrices are different for each element).

\subsection{Energy stability}
Let the energy in element $e$ be defined as
\begin{align}
  \label{eqn:eng:elm}
  \En^{e} = \frac{1}{2}
  \left(
  \vec{p}^{T}\mat{M}_{J}\vec{p}
  +
  \vec{v}_{x}^{T}\mat{M}_{J}\vec{v}_{x}
  +
  \vec{v}_{y}^{T}\mat{M}_{J}\vec{v}_{y}
  \right).
\end{align}
Then, the energy in the whole domain is
\begin{align}
  \En = \sum_{e\in\EE}\En^{e}.
\end{align}
If $\mat{M}_{J}$ is symmetric positive definite for all $e \in \EE$, it follows
that $\En$ is a well-defined norm of the solution. With this and the
restriction that $\mat{S}_{J}^{g}$ is positive for all mortar elements, we have
the following stability result.

\begin{theorem}\label{thm:stability}
  The semidiscrete scheme \eref{eqn:semi:p}--\eref{eqn:semi:vy} satisfies the
  energy estimate $\En(t) \le \En(0)$ for $t > 0$.
\end{theorem}
\begin{proof}
  See Appendix~\ref{app:thm:stability}.
\end{proof}

\subsection{Weight-Adjusted DG}\label{sec:WADG}
One of the computational challenges with \eref{eqn:semi:p}--\eref{eqn:semi:vy}
is that $\mat{M}_{J}$ must be inverted. In practice this means that different
factors of $\mat{M}_{J}$ will be needed for each element which drastically
increases the storage costs of the method. As will be seen, when the GD basis
is used on affine elements the mass matrix has a tensor product
(dimension-by-dimension) structure which allows for the efficient application of
its inverse.  However,
when the Jacobian determinant is non-constant this tensor product structure is
lost and the full two-dimensional mass matrix must be factored.

To overcome this computational challenge, we propose using the weight-adjusted
approach of  Chan, Hewett, and
Warburton~\cite{ChanHewettWarburton2017_curviWADG}. In this approach the mass
matrix is approximated as $\mat{M}_{J} \approx \mat{M}
\mat{M}_{1/J}^{-1}\mat{M}$.  To apply the inverse of this approximation, one
only needs to multiply by the mass matrix weighted by the $1/J$ and the inverse
of the reference element mass matrix $\mat{M}$. For GD elements the
action of $\mat{M}^{-1}$ can be efficiently applied in tensor product form;
efficiency on simplicial elements comes from the fact that they all have the
same reference mass matrix (and thus factors).

This weight-adjusted approximation of $\mat{M}_{J}$ is arrived at
by approximating multiplication by $J$ with the operator
$T^{-1}_{1/J}$:
\begin{align}
  \int_{\hat{e}}
   J \phi u
  \approx
  \int_{\hat{e}}
  \phi T^{-1}_{1/J} u.
\end{align}
Let $u \in V^{e}_{h}$, then $T^{-1}_{1/J} u \in V^{e}_{h}$ is defined by
\begin{align}
  \label{eqn:wadg:proj}
  \int_{\hat{e}}
  \phi \frac{1}{J} T^{-1}_{1/J} u
  =
  \int_{\hat{e}}
  \phi u,
\end{align}
for all $\phi \in V^{e}_{h}$. To see how this gives rise to the weight-adjusted
mass matrix above, we first define $u_{J} = T^{-1}_{1/J} u$ which allows us to
write
\begin{align}
  \label{eqn:wadg:mass}
  \int_{\hat{e}}
  \phi T^{-1}_{1/J} u
  =
  \vec{\phi}^{T} \mat{M} \vec{u}_{J}.
\end{align}
The quantity $u_{J}$ is calculated using \eref{eqn:wadg:proj}:
\begin{align}
  \vec{\phi}^{T} \mat{M}_{1/J} \vec{u}_{J} =
  \vec{\phi}^{T} \mat{M} \vec{u}.
\end{align}
Since $\phi \in V^{e}_{h}$ is arbitrary we then have that
\begin{align}
  \vec{u}_{J}
  =
  \mat{M}_{1/J}^{-1} \mat{M} \vec{u}.
\end{align}
Substituting this back into \eref{eqn:wadg:mass}, gives the weight-adjusted mass
matrix:
\begin{align}
  \int_{\hat{e}}
  \phi T^{-1}_{1/J} u
  =
  \vec{\phi}^{T} \mat{M}
  \mat{M}_{1/J}^{-1}
  \mat{M} \vec{u}.
\end{align}
Typically $\mat{M}_{1/J}$ will not be an exact mass matrix since
both $J$ and $1/J$ will be approximated and inexact quadrature is used.
With the assumption that $\mat{M}_{1/J}$ is symmetric positive definite, the
stability properties of the scheme remain unchanged.

We thus have the following corollary to Theorem~\ref{thm:stability}
\begin{corollary}
  The semidiscrete scheme \eref{eqn:semi:p}--\eref{eqn:semi:vy} with the
  weight-adjusted mass matrix satisfies the energy estimate $\En(t) \le
  \En(0)$ for $t > 0$.
\end{corollary}

\section{Galerkin Difference Basis}\label{sec:GDbasis}
Galerkin difference methods~\cite{BanksHagstrom2016JCP,
BanksHagstromJacangelo2018JCP}
are Galerkin finite elements methods built using basis functions defined on a
grid of degrees of freedom similar to a finite difference method. The key
feature of a GD basis is that the basis functions have compact support on the
grid, which results in banded element mass and stiffness matrices. We begin by
describing the GD approximation in one dimension and then discuss the
generalization to multiple dimensions.

\subsection{One-Dimensional Galerkin Difference Basis Functions}
Consider the domain $[-1,~1]$ discretized with $N+1$ grid equally spaced points.
The spacing between the grid points is $h = 2/N$ and the location of grid
point $i$ is $r_{i} = i h - 1$, for $i = 0,\dots, N$. Over each interval
$R_{i} = [r_{i},~r_{i+1}]$ a polynomial of degree $n = 2 m-1$ is
built using values at $2 m$ grid points centered around $R_{i}$; we call these
intervals subcells.

Consider a subcell $R_{i}$ sufficiently far from the boundary, i.e., $i \ge
m-1$ and $i \le N-m$. In the subcell it is natural to use a symmetric
stencil to construct the polynomial, namely values from grid points
$\{i+1-m,~i+2-m,~\dots,~i+m\}$. Thus, if $r \in R_{i}$ the GD
approximation of a function $\textup{u}$ would be
\begin{align}
  \label{eqn:ui}
  u(r) = \sum_{k = 1-m}^{m} u_{i+k}  \LL_{k}(r-ih),
\end{align}
where $u_{i+k}$ are the GD grid point values.  These values can be defined in a
number of ways, such as through interpolation $u_{i} = \textup{u}(r_{i})$ or an
$L^{2}$ projection of $\textup{u}$ onto the GD approximation space. Here,
$\LL_{k}(r)$ is the $n$th order Lagrange interpolating polynomial that satisfies
$\LL_{k}(jh) = \delta_{jk}$ for $j = 1-m, \dots, m$.
We say that a function of the form~\eref{eqn:ui} is in the space $\GD_{N,n}$
where $n = 2 m-1$.
This construction of the GD approximation implies that it is continuous and
its derivative is discontinuous between subcells.

Near the boundary there are several options for how to define the approximation.
One option is to use \eref{eqn:ui} but allow some degrees of
freedom to be outside of the domain. Namely, the left-most degree of freedom
would be $r_{1-m} = -1 - (m-1)h$ and the right-most $r_{N+m-1} = 1 +
(m-1)h$. This is the so-called \emph{ghost basis}
method~\cite{BanksHagstrom2016JCP}. Another option is to bias the
stencil toward the interior and use a non-symmetric stencil to construct the
approximation near the boundary. This can be done by either modifying the
interpolation formula for $u(r)$ near the boundary or by using~\eref{eqn:ui}
near the boundary with points outside the domain filled using extrapolation.
This second approaches is called the \emph{extrapolation}
method~\cite{BanksHagstrom2016JCP}. A third option would be to use a mixture of
the two methods with some nodes outside the domain being ghost basis nodes and
other nodes being extrapolated.

\begin{figure}
  \centering
  \includegraphics{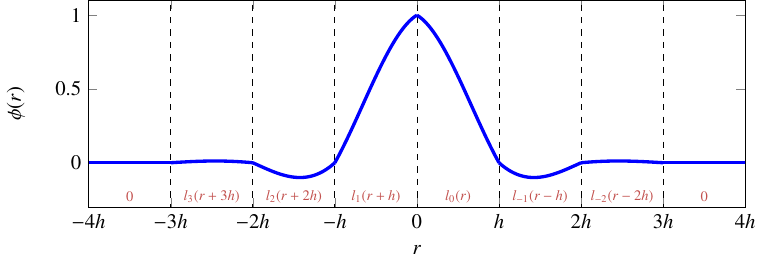}
  \caption{Example of $\phi(r)$ for $n = 5$ ($m = 3$). The equations at the base
    of the figure give the Lagrange polynomial used over the interval, which are
  denoted by the dashed lines.\label{fig:phi}}
\end{figure}
Regardless of whether the ghost basis or extrapolated boundary closure is used,
linearity and compactness of the interpolant allows us to define the global
interpolant as
\begin{align}
  \label{eqn:gdint}
  u = \sum_{k=1-m}^{N+m-1} u_{k} \phi_{k}.
\end{align}
Here the $k$th basis function is defined as $\phi_{k}(r) = \phi(r+1-kh)$,
where
\begin{align}
  \label{eqn:gd:basis}
  \phi(r) = \begin{cases}
    0,& \mbox{if } r \le -mh,\\
    \LL_{m  }(r +  m   h),  &\mbox{if }  -m h < r \le  -(m-1)h,\\
    \LL_{m-1}(r + (m-1)h),  &\mbox{if }  -(m-1)h < r \le -(m-2)h,\\
    \vdots\\
    \LL_{2-m}(r + (2-m)h),  &\mbox{if }  (m-2)h < r \le  (m-1)h,\\
    \LL_{1-m}(r + (1-m)h),  &\mbox{if }  (m-1)h < r \le  mh,\\
    0,& \mbox{if } mh < r.
  \end{cases}
\end{align}
The function $\phi$ has the Lagrange-like property $\phi(kh) = \delta_{0k}$
for $k = -m,\dots,m$ and is compactly supported such that $\phi(r) = 0$ for $|r|
> hm$.  An example of $\phi(r)$ for $n=5$ ($m = 3$) is shown in Figure~\ref{fig:phi}.
If vectors $\vec{\phi}$ and $\vec{u}$ are defined to be
\begin{align}
  \label{eqn:gd:stack}
  \vec{\phi}
  &=
  \begin{bmatrix}
    \phi_{1-m}\\
    \phi_{2-m}\\
    \vdots\\
    \phi_{N+m-2}\\
    \phi_{N+m-1}
  \end{bmatrix},
  &
  \vec{u}
  &=
  \begin{bmatrix}
    u_{1-m}\\
    u_{2-m}\\
    \vdots\\
    u_{N+m-2}\\
    u_{N+m-1}
  \end{bmatrix},
\end{align}
then interpolant \eref{eqn:gdint} can be rewritten as
\begin{align}
  u = \vec{\phi}^{T} \vec{u}.
\end{align}
In the case of the extrapolated boundary treatment, only a subset of the grid
points are actually stored/updated. Letting $\vec{\underbar{u}}$ be this subset,
then $\vec{u} = \mat{E}\vec{\underbar{u}}$ where $\mat{E}$ extrapolates the
boundary points and copies the interior points.

\begin{figure}
  \begin{center}
  \includegraphics{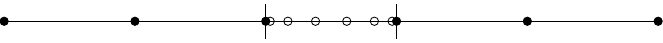}
  \end{center}
  \caption{The solid dots are the GD grid points (degrees of freedom) that would
  be used to build an interpolant over the center interval (indicated with the
  vertical lines) if $m = 3$, i.e., polynomial order over the interval is $n =
  5$. The open dots show the $6$ Legendre--Gauss nodes that would be
  used to integrate the inner product of two GD interpolants.\label{fig:grid}}
\end{figure}
In order to use the GD basis in the DG formulation, we need to be able to
both compute integrals and derivatives of GD based approximations. Since over
each subcell the GD approximation is polynomial, integrals can be simply
computed by using quadrature over each subcell. Namely, let the $N_{q}$
quadrature weights ${\{\omega_{k}\}}_{k=1}^{N_{q}}$ and points
${\{\zeta_{k}\}}_{k=1}^{N_{q}}$ approximate integrals over $[-1,\;1]$, then the
inner product of $u, v \in \GD_{N,n}$ is
\begin{align}
  \notag
  \int_{-1}^{1} v(r)u(r)\;dr
  &= \sum_{i=0}^{N-1} \int_{r_{i}}^{r_{i+1}} v(r)u(r)\;dr
  = \sum_{i=0}^{N-1} \frac{1}{h}\int_{-1}^{1}
  u\left(\frac{r_{i}+r_{i+1} + h\xi}{2}\right)v\left(\frac{r_{i}+r_{i+1} + h\xi}{2}\right)\;d\xi\\
  \label{eqn:gd:lg}
  &\approx \frac{1}{h} \sum_{i=0}^{N-1} \sum_{k=1}^{N_{q}}
  \omega_{k}
  v\left(\frac{r_{i}+r_{i+1} + h\xi_{k}}{2}\right)
  u\left(\frac{r_{i}+r_{i+1} + h\xi_{k}}{2}\right).
\end{align}
In our case we use a $2 m$-point Legendre--Gauss quadrature rule, which
integrates polynomials of degree $4 m+1$ exactly and thus the quadrature inner
product is exact (since over each subcell the degree of $uv$ is $2 n = 4 m - 2$).
An example for $n=5$ ($m = 3$) of the two grids (the GD grid and
Legendre--Gauss) quadrature grid is shown in Figure~\ref{fig:grid}.  The inner
product~\eref{eqn:gd:lg} can be written in matrix vector product form as
\begin{align}
  \int_{-1}^{1} vu
  = \vec{v}^{T}\mat{\bar{L}}^{T} \mat{\bar{W}} \mat{\bar{L}}\vec{u}
  = \vec{v}^{T}\mat{\bar{M}}\vec{u};
\end{align}
the overbar accent is used here to highlight that these are matrices for the
one-dimensional basis functions.  Here the diagonal matrix of composite
quadrature weights is $\mat{\bar{W}} = (\mat{I}_{N} \otimes \mat{\omega})/h$
with $\otimes$ being the matrix Kronecker product, $\mat{I}_{N}$ the $N \times
N$ identity matrix, and $\mat{\omega}$ the diagonal matrix of
${\{\omega_{k}\}}_{k=1}^{N_{q}}$. The Matrix $\mat{\bar{L}}$ is the
interpolation matrix from GD grid points to the intra-subcell quadrature points.
Regardless of the boundary closure, the mass matrix $\mat{\bar{M}}$ is banded
with bandwidth $m+1$.

In order to take derivatives of GD based solutions, we first note that even if
$u \in \GD_{N,n}$ in general $u' \notin \GD_{N,n}$. That said, we can exploit the
fact that over each subcell $u$ is polynomial. Namely, in order to take
derivatives we first interpolate the solution to the previous quadrature grid
and then take derivatives discretely at the quadrature grid. Namely, we define
the GD derivative matrix
\begin{align}
  \mat{\bar{D}} = h (\mat{I_{N}} \otimes \mat{D}_{q}) \mat{L}
\end{align}
where $\mat{D}_{q}$ is the polynomial derivative matrix for a function defined at
the quadrature nodes ${\{\zeta_{k}\}}_{k=1}^{N_{q}}$; scaling by $h$ arises
because the quadrature nodes are for the interval $[-1,~1]$. We note that the
stiffness matrix
\begin{align}
  \mat{\bar{S}} = \mat{\bar{L}}^T \mat{\bar{W}} \mat{\bar{D}}
\end{align}
will be exact as long as the quadrature rule used over the subcells is of order
$2 n -1$. The stiffness matrix $\mat{\bar{S}}$ has the same banded structure as
$\mat{\bar{M}}$.

\subsection{Two-Dimensional Galerkin Difference Operators}\label{sec:gd:2d}
The one-dimensional GD basis and operators can be extended to multiple dimensions
via the tensor product. Namely, let the two-dimensional domain $\hat{\Omega} =
[-1,~1] \times [-1,~1]$ be discretized using $(N_{r}+1) \times (N_{s}+1)$
interior points, so that the grid points are $(r_{i}, s_{j}) = (i h_{r} - 1, j
h_{s} - 1)$ with $h_{r} = 2 / N_{r}$ and $h_{s} = 2 / N_{s}$.  Over each subcell
$R_{ij} = [r_{i},~r_{i+1}] \times
[s_{j},~s_{j+1}]$ an $n = 2 m-1$ degree tensor product polynomial is defined
using nearby grid points. Namely, for $(r,s) \in R_{ij}$ the GD interpolation of
$\textup{u}$ would be
\begin{align}
  \label{eqn:urs}
  u(r,s) = \sum_{j=1-m}^{m}\;\sum_{l=1-m}^{m} u_{i+j,j+l} \LL_{j}(r-ih_{r})
  \LL_{l}(s-jh_{s}),
\end{align}
where $\LL_{k}$ and $\LL_{l}$ are the previously defined one-dimensional Lagrange
polynomials and $u_{i,j}$ are the GD grid points values (see~\eref{eqn:ui}); we say
that a function in the form of~\eref{eqn:urs} is in the space $\GD_{N_{r} \times
N_{s},n} = \GD_{N_{r},n}
\times \GD_{N_{s},n}$.  With this tensor product form, the approximation over each
cell $R_{ij}$ is constructed using a box of ${(2 m)}^{2}$ grid points near the cell,
grid points $(r_{i+1-m},r_{j+1-m})$ through $(r_{i+m}, s_{j+m})$. Using the
previously defined GD basis functions $\phi_{k}$, see \eref{eqn:gdint} and
\eref{eqn:gd:basis}, the approximation over the whole domain can be written as
\begin{align}
  u(r,s) = \sum_{i = 1-m}^{N_{r}+m-1} \sum_{j = 1-m}^{N_{s}+m-1} u_{i,j}
  \phi_{i}(r;\;h_{r}) \phi_{j}(s;\;h_{s});
\end{align}
here the parameter $h$ has been added to $\phi_{i}(r;\;h)$ since each dimension
is allowed to have a different grid spacing.  In two dimensions, the grid function is
denoted with the vector
\begin{align}
  \vec{u} =
  \begin{bmatrix}
    u_{-g,-g} & u_{-g,-g+1} & u_{-g,-g+2} & \cdots & u_{N_{r}+g,N_{s}+g}
  \end{bmatrix}^{T}.
\end{align}
With this, the interpolant can be written as
\begin{align}
  u(r,s) = {\left(\vec{\phi}_{s}(s) \otimes \vec{\phi}_{r}(r)\right)}^{T}
  \vec{u},
\end{align}
with $\otimes$ being the matrix Kronecker product and $\vec{\phi}_{r}$ and
$\vec{\phi}_{s}$ being the stacking of basis functions
${\{\phi_{k}(r;\;h_{r})\}}_{k=1-m}^{N_{r}+m-1}$ and
${\{\phi_{k}(r;\;h_{s})\}}_{k=1-m}^{N_{s}+m-1}$.

The two-dimensional operators needed to define the DG scheme over a GD element
are then
\begin{align}
  \mat{W} &= \mat{\bar{W}}_{s} \otimes \mat{\bar{W}}_{r},&
  \mat{L} &= \mat{\bar{L}}_{s} \otimes \mat{\bar{L}}_{r},\\
  \mat{D}_{r} &= \mat{\bar{L}}_{s} \otimes \mat{\bar{D}}_{r},&
  \mat{D}_{r} &= \mat{\bar{D}}_{s} \otimes \mat{\bar{L}}_{r},
\end{align}
where the subscripts $r$ and $s$ on the one-dimensional GD matrices indicate
that these are define with respect to grids of size $N_{r}$ and $N_{s}$,
respectively.

The GD reference mass matrix is
\begin{align}
  \mat{M} &= \mat{L}^{T} \mat{W} \mat{L}
  =
  (\mat{\bar{L}}_{s}^{T} \mat{\bar{W}}_{s} \mat{\bar{L}}_{s})
  \otimes
  (\mat{\bar{L}}_{r}^{T} \mat{\bar{W}}_{r} \mat{\bar{L}}_{r})
  =
  \mat{\bar{M}}_{s}
  \otimes
  \mat{\bar{M}}_{r}.
\end{align}
Since the reference mass matrix has a tensor product structure its inverse will
as well:
\begin{align}
  \mat{M}^{-1}
  =
  \mat{\bar{M}}_{s}^{-1}
  \otimes
  \mat{\bar{M}}_{r}^{-1}.
\end{align}
This means that when $\mat{M}^{-1}$ is applied, it can be done
dimension-by-dimension using a banded Cholesky factorization of
$\mat{\bar{M}}_{s}$ and $\mat{\bar{M}}_{r}$. Unfortunately, when the
element is not affine the Jacobian determinant weighted mass matrix does not
have this tensor product structure and the full factorization would be required
to apply $\mat{M}_{J}^{-1}$, thus motivating the use of the weight-adjusted
methodology.

\subsection{Construction of Metric Terms}\label{sec:GD:metric}
In the DG scheme~\eref{eqn:semi:p}--\eref{eqn:semi:vy} metric terms are required
at the quadrature nodes. Of course one option is to store the metric terms (or
their approximation) at the quadrature nodes. In the case of GD though, this
greatly increases the storage since the ratio of quadrature nodes to GD grid
points is $4 m^2$ in two-dimensions (assuming that the basis of the
one-dimensional quadrature is a $2 m$-point Legendre--Gauss quadrature rule).
Thus, for a practical implementation of GD, we suggest storing the volume metric
terms at the GD interpolation nodes and interpolating them on-the-fly to the
quadrature nodes; of course another option, if the exact metric terms were
available, would be to compute them on-the-fly at the quadrature nodes. The
following construction, except where noted, holds for both GD and
simplicial elements.

Let $x, y \in V^{e}_{h}$, we then define the metric derivatives as the
$L^2$-projection of the exact derivatives into the approximation space.  For
instance, we let the function $x_r \in V^{e}_{h}$ be such that
\begin{align}
  \label{eqn:xr:int}
  \int_{-1}^{1}\int_{-1}^{1} v \left(x_{r} - \frac{dx}{dr}\right) = 0
\end{align}
for all $v \in V^{e}_{h}$; an analogous construction is
used to define $x_s, y_r, y_s \in V^{e}_{h}$.
In matrix-vector form all of the metric derivatives are then
\begin{align}
  \label{eqn:xr}
  \vec{x}_{r}
  &=
  \mat{M}^{-1} \mat{L}^{T} \mat{W} \mat{D}_{r} \vec{x}
  ,&
  \vec{x}_{s}
  &=
  \mat{M}^{-1} \mat{L}^{T} \mat{W} \mat{D}_{s} \vec{x}
  ,\\
  \vec{y}_{r}
  &=
  \mat{M}^{-1} \mat{L}^{T} \mat{W} \mat{D}_{r} \vec{y}
  ,&
  \label{eqn:ys}
  \vec{y}_{s}
  &=
  \mat{M}^{-1} \mat{L}^{T} \mat{W} \mat{D}_{s} \vec{y}.
\end{align}
For GD elements these reduce to
\begin{align}
  \vec{x}_{r}
  &=
  \left(\mat{I}_{N_{s}}
  \otimes
  \left(\mat{\bar{M}}_{r}^{-1} \mat{\bar{L}}_{r}^{T} \mat{\bar{W}}_{r}
  \mat{\bar{D}}_{r}\right)\right) \vec{x}
  ,&
  \vec{x}_{s}
  &=
  \left(\mat{\bar{M}}_{s}^{-1} \mat{\bar{L}}_{s}^{T} \mat{\bar{W}}_{s}
  \mat{\bar{D}}_{s}\right)
  \otimes
  \left(\mat{I}_{N_{r}}
  \right) \vec{x},\\
  \vec{y}_{r}
  &=
  \left(\mat{I}_{N_{s}}
  \otimes
  \left(\mat{\bar{M}}_{r}^{-1} \mat{\bar{L}}_{r}^{T} \mat{\bar{W}}_{r}
  \mat{\bar{D}}_{r}\right)\right) \vec{y},&
  \vec{y}_{s}
  &=
  \left(\mat{\bar{M}}_{s}^{-1} \mat{\bar{L}}_{s}^{T} \mat{\bar{W}}_{s}
  \mat{\bar{D}}_{s}\right)
  \otimes
  \left(\mat{I}_{N_{r}}
  \right) \vec{y},
\end{align}
which means that for GD these calculations can be performed along the grid
lines. We then define the metric terms needed at the quadrature nodes as the
interpolation of these metric derivatives. Namely, we define
\begin{align}
  \label{eqn:Jrx}
  \mat{J}\mat{r}_{x} &= \mbox{diag}\left( \mat{L}\vec{y}_{s}\right),&
  \mat{J}\mat{r}_{y} &= \mbox{diag}\left(-\mat{L}\vec{x}_{s}\right),&
  \mat{J}\mat{s}_{x} &= \mbox{diag}\left(-\mat{L}\vec{y}_{r}\right),&
  \mat{J}\mat{s}_{y} &= \mbox{diag}\left( \mat{L}\vec{x}_{r}\right),
\end{align}
where the operator $\mbox{diag}(\cdot)$ turns a vector into a diagonal matrix.
When needed, the Jacobian determinant can be computed from these interpolated
values as:
\begin{align}
  \label{eqn:J:approx}
  \mat{J} = \left(\mat{J}\mat{r}_{y}\right) \left(\mat{J}\mat{s}_{x}\right) -
  \left(\mat{J}\mat{r}_{x}\right) \left(\mat{J}\mat{s}_{y}\right).
\end{align}

An important property of our approach to computing the metric terms is that we
preserve discretely the divergence theorem, which will enable us to show that
the scheme can be made both conservative and constant preserving.
\begin{theorem}\label{thm:GD:div}
  An element with metric terms computed using~\eref{eqn:Jrx}
  with~\eref{eqn:xr}--\eref{eqn:ys} satisfies
  \begin{align}
    \vec{1}^{T} \mat{S}_{x} \vec{v}_{x} + \vec{1}^{T} \mat{S}_{y} \vec{v}_{y}
    =
    \int_{\hat{\partial e}} S_{J} \left(n_{x} v_{x} + n_{y} v_{y}\right),
  \end{align}
  for all $v_{x}, v_{y} \in V^{e}_{h}$ where the surface
  Jacobian $S_{J}$ and outward normal ${[n_{x}\;n_{y}]}^{T}$ are calculated
  using $x, y \in V^{e}_{h}$; here $\vec{1}$ is the vector
  of ones.
\end{theorem}
\begin{proof}
  See Appendix~\ref{app:thm:GD:div}
\end{proof}

Theorem~\ref{thm:GD:div} alone is not enough to guarantee a conservative and
constant preserving scheme. In concert with Theorem~\ref{thm:GD:div} we need the
following consistency statement about the mortar and element surface integrals:
\begin{align}
  \label{eqn:surf:mort:int}
  \int_{\hat{\partial e}} S_{J} \left(n_{x} v_{x} + n_{y} v_{y}\right)
  =
  \sum_{g \in \GG^{e}}
  \vec{1}^{T} {(\mat{L}^{g})}^{T} \mat{S}_{J}^{g} \mat{W}^{g} \vec{v}_{n}^{-},
\end{align}
for all $v_{x}, v_{y} \in V^{e}_{h}$.  With this we now have the following
result concerning the conservation and constant preserving properties of the
scheme.
\begin{theorem}\label{thm:conserved:constant}
  If all of the elements $e\in\EE$ satisfy both Theorem~\ref{thm:GD:div}
  and~\eref{eqn:surf:mort:int}, then
  scheme~\eref{eqn:semi:p}--\eref{eqn:semi:vy} is both conservative and constant
  preserving with periodic boundary conditions.
\end{theorem}
\begin{proof}
  See Appendix~\ref{app:thm:conserved:constant}
\end{proof}
For the two-dimensional problems considered here, the assumption
of~\eref{eqn:surf:mort:int} holds if:
\begin{enumerate}
  \item the surface Jacobian determinant times the components of the unit normal
    on $g$ is $S_{J}^{g}n_{x}^{-}, S_{J}^{g}n_{y}^{-} \in \mathbb{P}_{n}$ (e.g.,
    a polynomial in the mortar space);
  \item the quadrature on the mortar can integrate polynomials of degree $2 n$,
    (e.g., the right-hand side of~\eref{eqn:surf:mort:int} is exact); and
  \item on both sides of each mortar element $g$, the physical coordinates are
    discretely conforming, e.g., $x^{-} = x^{+}$ and $y^{-} = y^{+}$ along the
    interface.
\end{enumerate}
The implications of point 3 for GD elements comes from the fact that
the only function that is in both $\GD_{N_{1},n}$ and $\GD_{N_{2},n}$ for $N_{1} \ne
N_{2}$ are functions in $\mathbb{P}_{n}$. Thus, along nonconforming GD
interfaces, the coordinate transform must be a single polynomial along the
entire interface. For purely computational interfaces, this is not much of a
limitation (since interfaces between elements are artificial). For problems with
physical interfaces, such as a layered discontinuous materials (not considered
here), resolving the geometry more accurately requires increasing the number of
elements along the interface not (only) increasing the number of grid points
inside the elements. Of course, these are only considerations if conservation
and constant preserving properties are needed. Additionally, when GD elements
are coupled with simplicial elements this is likely not a constraint, since the
GD elements would likely be affine with any complexity in the geometry handled
using the more flexible curved simplicial elements.

For three-dimensional problems, more restrictions on the geometry are needed to
insure the conservation and constant preserving properties. Namely, in both the
conforming and nonconforming cases care is needed to insure that certain
metric products are in the approximation space; see for
instance~\cite{Kopriva2006jsc} for the conforming case
and~\cite{KozdonWilcox2018} for the nonconforming case.

\section{Results}
In this section we test the properties of the above defined discretization. The
tests are broken into two sets. We begin by testing the weight-adjusted
discontinuous Galerkin method (WADG) with the GD basis. We test the accuracy of
the WADG projection for the GD basis and then test the constant-preserving,
conservation, and accuracy properties of the scheme on nonconforming
curved GD meshes. We then move on to the coupling of GD elements with simplicial
elements, first comparing the time step restrictions of each of the methods and
then moving on to the accuracy of the coupled scheme; in all the coupled test
problems the GD order and polynomial order used on simplicial elements are chosen
to be the same.

In all tests, the Taylor time integration is used to advance the
semidiscrete scheme. Namely, in all cases the semidiscrete scheme can be written
as
\begin{align}
  \frac{d \vec{w}}{d t} &= \mat{A} \vec{w},
\end{align}
with $\vec{w}$ being the solution vector and $\mat{A}$ being the spatial
discretization. The $k^{th}$ order accurate Taylor time stepping scheme is
\begin{align}
  \mat{w}(t+\Delta t) \approx \sum_{l=0}^{k} \frac{\Delta t^{l}}{l!}\mat{A}^{l}
  \mat{w}(t),
\end{align}
where $\Delta t$ is the time step size. The scheme is locally stable (e.g., the
stability region crosses the imaginary axis) for orders $k=4l-1$ and $k = 4l$
with $l \in \mathbb{Z}^{+}$.  The temporal order $k$ used is the minimum local
stable $k$ such that $k > n$, e.g., the temporal order both locally stable and at
least one order higher than the spatial order.

Throughout, error in the solution is measured using the energy
norm
\begin{align}
  \label{eqn:error}
  \mbox{error} &=
  \sqrt{
  \frac{1}{2}
  \sum_{e\in\EE}
  \int_{\hat{e}} J \left(
  \Delta p^{2}
  +
  \Delta v_{x}^{2}
  +
  \Delta v_{y}^{2}
  \right)},
\end{align}
where $\Delta p$, $\Delta v_{x}$, and $\Delta v_{y}$ are the difference between
the computed and exact solutions. The integrals in~\eref{eqn:error} are
approximated using the element quadrature rule, with differences between the
computed and exact solutions evaluated at the quadrature points.
Results are reported with respect to refinement level. For GD elements
refinement is done with grid doubling (e.g., the 1-D grid spacing is cut in half
with each level of refinement). For the simplicial elements refinement is done
by quadrisection (e.g., each reference triangle is split into four similar
triangles).

The codes used in this paper are freely available at
\url{https://github.com/bfam/GDComplexGeometries}.  The simplicial element
scheme is implemented using MATLAB codes from~\cite{HesthavenWarburton2008}.

\subsection{WADG-GD Projection Accuracy}
We begin by testing the accuracy of the weight-adjusted approximation for GD
elements as well as the storage of the metric terms at the GD interpolation
nodes. Namely, we want to quantify the error made in the approximation of the
$L^{2}$-projection of a function $f(x,y)$ into the curved-GD approximation
space. Following~\cite{ChanEvans2018}, we consider the coordinate
transform
\begin{align}
  \label{eqn:proj:coord}
  x &= r + \beta \cos\left(\frac{3 \pi s}{2}\right)\cos\left(\frac{\pi r}{2}\right),&
  y &= s + \beta \sin\left(\frac{3 \pi r}{2}\right)\cos\left(\frac{\pi s}{2}\right).
\end{align}
Here the domain is $-1 \le r,s \le 1$ and $\beta$ is a parameter which controls
the regularity of the transform. We let there be a single GD element
with an interior grid of $(N+1) \times (N+1)$ points, e.g., the GD grid spacing
is $h = 2/N$ in each dimension.

The $L^{2}$-projection of a function $f(x,y)$ into the GD approximation space is
defined to be the function $\bar{f}\in V_{h}$ such that
\begin{align}
  \int_{\hat{\Omega}} \phi J \left(f - \bar{f}\right) = 0, \qquad \forall \phi \in
  V_{h}.
\end{align}
Here $V_{h} = \GD_{N \times N,n}$ and the Jacobian determinant is $J$.  Using
GD quadrature to approximate the integrals gives the matrix problem:
\begin{align}
  \label{eqn:WA:L2:J:proj}
  \mat{M}_{J} \bar{\vec{f}} &= \mat{L}^{T} \mat{W} \vec{f}_{q},
\end{align}
where $\vec{f}_{q}$ is the exact $f$ evaluated at the quadrature nodes and
$\bar{\vec{f}}$ is the projection of $f$ into the GD space.
We define the $L^2$-error in the approximation as
\begin{align}
  \label{eqn:WA:eL2:J}
  e_{L2} &= \int_{\hat{\Omega}} J {\left(f - \bar{f}\right)}^{2} \approx
  \vec{\Delta}_{q}^{T} \mat{W} \mat{J} \vec{\Delta}_{q},
  \quad
  \vec{\Delta}_{q} = \vec{f}_{q} - \mat{L}\bar{\vec{f}}.
\end{align}
The error $e_{L2}$ depends on $f$, the mesh skewness parameter
$\beta$, the number of points $N$, the GD approximation order $n$, the boundary
treatment, and any approximations used for $\mat{M}_{J}$. Three different
approaches to handling $\mat{M}_{J}$ will be considering:
\begin{itemize}
  \item $L^{2}$-projection: the mass matrix $\mat{M}_{J} = \mat{L}^{T} \mat{W}
    \mat{J} \mat{L}$ with the exact Jacobian determinant evaluated at the
    quadrature points;
  \item WAGD\@: weight-adjusted approximation of the mass matrix with
    $\mat{M}_{1/J} = \mat{L}^{T} \mat{W} \mat{J}^{-1} \mat{L}$ evaluated using
    the exact Jacobian determinant evaluated at the quadrature points; and
  \item inexact WAGD\@: weight-adjusted approximation of the mass matrix with
    $\mat{M}_{1/J}$ evaluated using $\mat{J}$ as defined by~\eref{eqn:J:approx}
    with the coordinate points $x$ and $y$ projected into the GD space using
    the same quadrature rule, e.g., $\mat{M} \bar{\vec{x}} = \mat{L}^{T} \mat{W}
    \vec{x}$ with $\mat{x}$ being $x$ evaluated at the quadrature nodes
    and $\mat{M}$ being the reference GD mass matrix.
\end{itemize}
In all cases the exact Jacobian determinant is used when evaluating the right-hand side of
\eref{eqn:WA:L2:J:proj} and in the computation of the error~\eref{eqn:WA:eL2:J}.
For all tests, the base mesh uses $n = N$.

\begin{figure}[tb]
  \centering
  \begin{minipage}{0.48\textwidth}
  \includegraphics{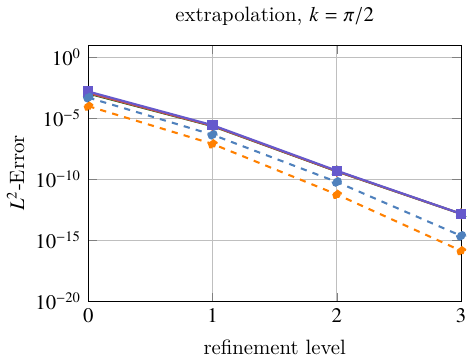}
  \end{minipage}
  \hfill
  \begin{minipage}{0.48\textwidth}
  \includegraphics{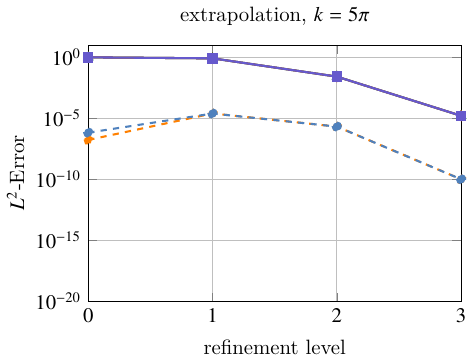}
  \end{minipage}\\
  \begin{minipage}{0.48\textwidth}
  \includegraphics{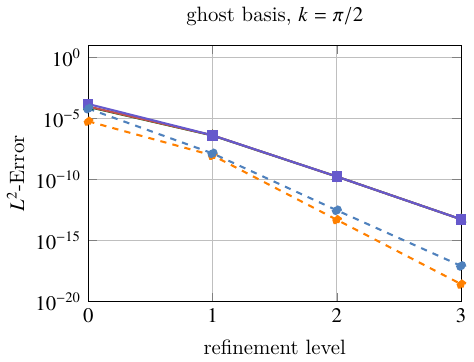}
  \end{minipage}
  \hfill
  \begin{minipage}{0.48\textwidth}
  \includegraphics{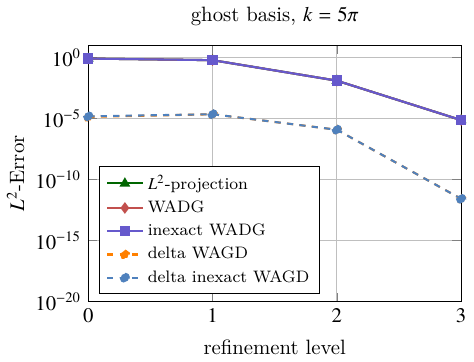}
  \end{minipage}
  \caption{Error in representing $f(x,y) = \cos(kx)\cos(ky)$ with a GD
  approximation ($n=5$) on a grid defined by \eref{eqn:proj:coord} with
  $\beta=1/8$. The base mesh used $(n+1)\times(n+1)$ interior GD
  nodes.\label{fig:wadg:err:b1}}
\end{figure}
\begin{figure}[tb]
  \centering
  \begin{minipage}{0.48\textwidth}
  \includegraphics{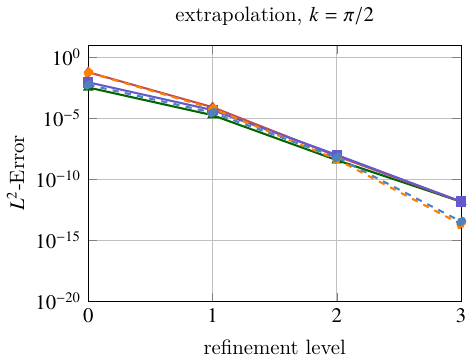}
  \end{minipage}
  \hfill
  \begin{minipage}{0.48\textwidth}
  \includegraphics{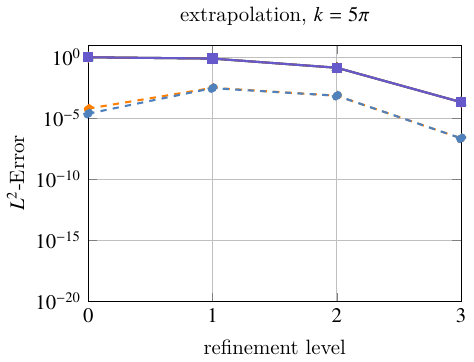}
  \end{minipage}\\
  \begin{minipage}{0.48\textwidth}
  \includegraphics{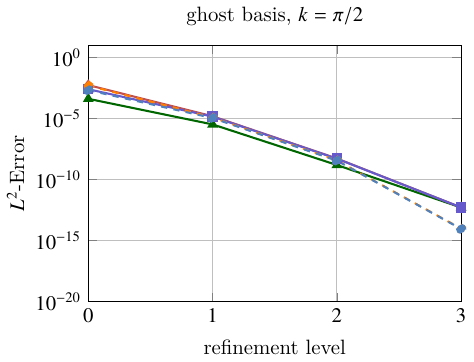}
  \end{minipage}
  \hfill
  \begin{minipage}{0.48\textwidth}
  \includegraphics{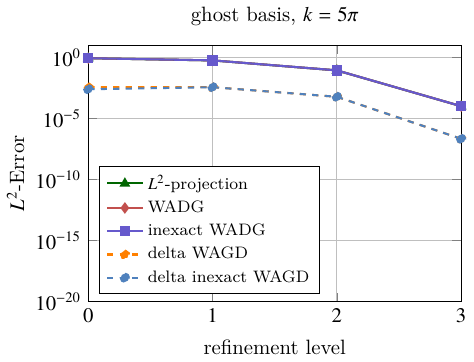}
  \end{minipage}
  \caption{Error in representing $f(x,y) = \cos(kx)\cos(ky)$ with a GD
  approximation ($n=5$) on a grid defined by \eref{eqn:proj:coord} with
  $\beta=0.22$. The base mesh used $(n+1)\times(n+1)$ interior GD
  nodes.\label{fig:wadg:err:b2}}
\end{figure}
Figures~\ref{fig:wadg:err:b1} and~\ref{fig:wadg:err:b2} show the results for
$\beta = 1/8$ and $\beta=0.22$, respectively, using both the extrapolation (top
row) and ghost basis (bottom row) method with polynomial order $n=5$. The
function being approximated in both cases is $f(x,y) = \cos(kx)\cos(ky)$ with
$k=\pi/2$ (left column) and $k=5\pi$ (right column); here the lower the value of
$k$ the smoother the function and the better the overall approximation will be.
In the figures the $L^{2}$-error for the $L^{2}$-projection, WADG, and inexact
WADG are all plotted, but are not visible as they lie on top of one another.
Therefore we also show the difference between WAGD and the
$L^{2}$-projection (labeled \emph{delta WADG}) as well as the difference between
inexact WAGD and the $L^{2}$-projection (labeled \emph{delta inexact WADG}).

As can be seen both WADG and inexact WADG are accurate, and in most
cases the error is on the same order as the $L^{2}$-projection;
this is indicated by the fact that the delta lines are below the
$L^{2}$-projection lines. In all cases the weight-adjusted approaches are
high-order accurate and storing the metric terms at the GD interpolation nodes
does not significantly impact the accuracy of the weight-adjusted approach. As
the theory
of~\cite{ChanHewettWarburton2017_curviWADG} suggests, the weight-adjusted
approaches does better with more regular geometry mappings (lower $\beta$
values), which explains the fact that at lower resolutions the delta lines are
above the $L^{2}$-projection curves for $k=\pi/2$ and $\beta = 0.22$.  In all
the tests the ghost basis method outperforms the extrapolation method for the
same grid spacing, though it is worth noting that for the same grid spacing
the ghost basis method has $\sim 4 (n-1) N$ more degrees of freedom than the
extrapolation method.

Though not shown, we note that with $n=3$ and $\beta=0.22$ the approximated
Jacobian determinant had a negative value when $N = 6$, which shows that storage
of the Jacobian determinant at the interpolation nodes requires that the
geometry approximation be sufficiently well-resolved for this storage strategy
to work in practice.

One approach that could be used if under-resolved geometries were needed would
be to store the square root the Jacobian determinant at the GD interpolation
nodes. Namely, first the square root of the Jacobian determinant is computed at
the  GD quadrature nodes:
\begin{align}
  \vec{\tilde{J}}_{sq} = \sqrt{
    \left(\mat{D}_r \vec{x}\right) \odot \left(\mat{D}_s \vec{y}\right)
    -
    \left(\mat{D}_s \vec{y}\right) \odot \left(\mat{D}_r \vec{y}\right)
  },
\end{align}
where $\odot$ is the Hadamard (componentwise) product of two vectors and the
square root is applied componentwise. The square root of the Jacobian
determinant at the quadrature nodes $\vec{\tilde{J}}_{sq}$ is then projected
to the GD interpolation nodes for storage:
\begin{align}
  \vec{J}_{sq} = \mat{M}^{-1} \mat{L}^{T} \mat{W} \vec{\tilde{J}}_{sq}.
\end{align}
When the Jacobian determinant is needed at the quadrature nodes the square root
of the Jacobian determinant is interpolated and squared:
\begin{align}
  \label{eqn:Jsq:approx}
  \vec{J} = {\left(\mat{L} \vec{J}_{sq}\right)}^{2},
\end{align}
where the square is applied componentwise; the form of \eref{eqn:Jsq:approx} is
a replacement for \eref{eqn:J:approx} which is the computation of the Jacobian
determinant from the metric derivatives after they have been interpolated to the
quadrature nodes. Tests (not shown) have verified that this approach seems to
have minimal impact on the accuracy of the scheme when the solution is
well-resolved.

\subsection{Constant Preservation, Conservation, and Energy Stability of WADG-GD}
\begin{figure}
  \centering
  \includegraphics{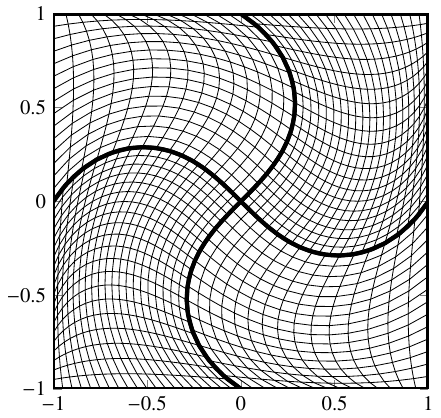}
  \caption{Skewed mesh used for the constant preservation and conservation test
  problem defined by global
  transformation~\eref{eqn:global:tran}.\label{fig:const:mesh}}
\end{figure}
Here we examine the constant-preserving and conservation properties of
WADG-GD\@ with the computation of metric terms given in
Section~\ref{sec:GD:metric}. In this test we consider the domain defined by the
transformation
\begin{align}
  \label{eqn:global:tran}
  x &=  r \cos(\beta) + s \sin(\beta), &
  y &= -r \sin(\beta) + s \cos(\beta), &
  \beta = \frac{\pi}{4}\left(1-r^2\right) \left(1-s^2\right).
\end{align}
Here the domain is $-1 \le r,s \le 1$.  The domain is discretized using
four GD elements as shown in Figure~\ref{fig:const:mesh} using
$16\times16$ (interior) GD mesh in the bottom-right and top-left elements and a
$21\times21$ (interior) GD mesh for the other two elements.

We consider two approaches to handling the geometry. In the first case
we project the transformation~\eref{eqn:global:tran} into the GD space; the
$L^{2}$-projection is approximated using the GD quadrature grid.
Across non-conforming interfaces the mesh will be discretely discontinuous
because the elements on either side of the interface have different
representations of the geometry. This means that the mortar element based metric
terms will be different than (at least one of) the GD element metric terms; in our
case we calculate metric terms on the mortar by averaging $x$ and $y$ values
from the two sides of each mortar element at the mortar element quadrature
nodes. We call this first approach the \emph{discontinuous geometry
approximation}.  Note that even if the mesh is conforming, the
$L^{2}$-projection can cause the geometry to be discretely discontinuous if
continuity is not enforced between elements.

In the second case, the mesh is made watertight by modifying the $x$ and $y$
values at the GD interpolation nodes along the non-conforming interfaces.
Namely, the grid values along the interface are replaced with values in
$\mathbb{P}_{n}$. These values are calculated along each interface using an
$L^{2}$-projection which preserves corner values. For instance, the points along
the interface $-1 \le r \le 1$ and $s = -1$ are $\bar{x}, \bar{y} \in
\mathbb{P}_{n}$ which satisfy for all $\phi\in\mathbb{P}_{n}$:
\begin{align}
  \int_{-1}^{0} \phi(r) (x(r,-1) -\bar{x}(r)) = 0,
  ~\bar{x}(-1) = x(-1,-1),
  ~\bar{x}(0) = x(0,-1),\\
  \int_{-1}^{0} \phi(r) (y(r,-1) -\bar{y}(r)) = 0,
  ~\bar{y}(-1) = y(-1,-1),
  ~\bar{y}(0) = y(0,-1).
\end{align}
Note that values at nodes not on the interface are not modified by
this procedure. On both sides of the interface the projections are approximated
using $21$ Chebyshev nodes of the second kind, and thus both sides of the mortar
elements use the same polynomial along the interface. Since the coordinate
points on the interface are the same for both sides of the element,
constraint~\eref{eqn:surf:mort:int} is satisfied and
Theorem~\ref{thm:conserved:constant} applies. We call this second approach the
\emph{watertight geometry approximation}.

\begin{figure}
  \centering
  \begin{minipage}{0.4\textwidth}
  \includegraphics{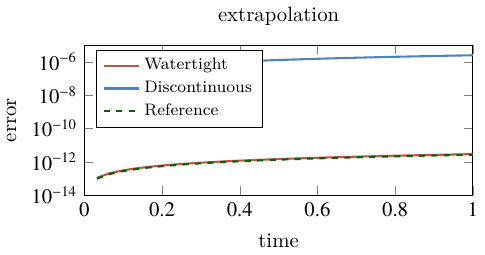}
  \end{minipage}
  \hfill
  \begin{minipage}{0.4\textwidth}
  \includegraphics{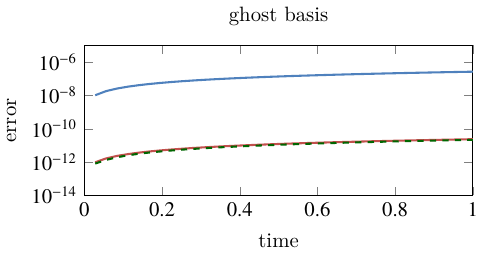}
  \end{minipage}
  \caption{Error versus time for the constant preservation test using mesh
  given in Figure~\ref{fig:const:mesh} with $n=7$.\label{fig:constant:preserving}}
\end{figure}
In order to test the ability of the scheme to preserve constants, we use the
initial condition $v_{x} = 1$, $v_{y} = 2$, and $p = 3$ with periodic boundary
conditions. If the scheme preserved constants the initial values will not
change with time, in other words the error in the solution should be
(approximately) zero. In Figure~\ref{fig:constant:preserving} the error for both
the watertight and discontinuous geometry approximations are shown for both the
extrapolation and ghost basis GD methods with $n=7$. Also included in the figure
the results from a conforming reference mesh. In the conforming reference case
each GD element is discretized using a $21\times21$ (interior) GD mesh; note
that since the $L^2$-projection of the coordinate transformation is done element by
element the geometry will be discontinuous, thus to make the geometry watertight
in the conforming case we interpolate the transformation along the boundaries of the
elements. As can be seen, the discontinuous geometry approximation has error
growth with time whereas the watertight mesh matches the conforming reference
case.

\begin{figure}
  \centering
  \begin{minipage}{0.4\textwidth}
  \includegraphics{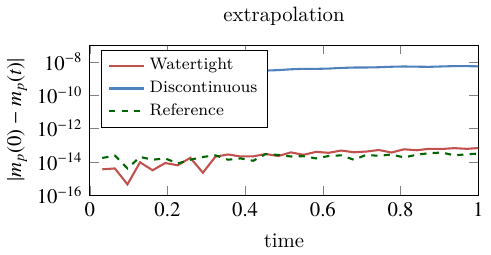}
  \end{minipage}
  \hfill
  \begin{minipage}{0.4\textwidth}
  \includegraphics{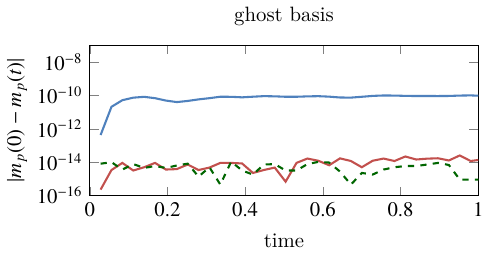}
  \end{minipage}
  \caption{Change in $m_p$ versus time for the conservation test using mesh
  given in Figure~\ref{fig:const:mesh} with $n=7$.\label{fig:conservation}}
\end{figure}
In order to test the conservation properties of the scheme, we use a
pseudorandom initial condition with periodic boundary conditions. We then define
the mass in each component of the solution at time $t$ to be
\begin{align}
  m_{v_{x}} &= \vec{1}^{T} \mat{M} \mat{M}_{1/J}^{-1} \mat{M} \vec{v_{x}},&
  m_{v_{y}} &= \vec{1}^{T} \mat{M} \mat{M}_{1/J}^{-1} \mat{M} \vec{v_{y}},&
  m_{p}     &= \vec{1}^{T} \mat{M} \mat{M}_{1/J}^{-1} \mat{M} \vec{p}.
\end{align}
We highlight that here we are not using $\mat{M}_{J}$ but the WADG-GD mass
matrix because it is with respect to this matrix that the scheme is
conservative. In Figure~\ref{fig:conservation} the absolute difference
$|m_{p}(t)-m_{p}(0)|$ is given for $n=7$ for the extrapolation and ghost basis
GD methods; mass components $m_{v_{x}}$ and $m_{v_{y}}$ are conserved even in
the discontinuous geometry case due to the use of the weak derivative (see
Appendix~\ref{app:thm:conserved:constant}). As can be seen the discontinuous
geometry is not conservative, whereas the watertight mesh has the same
conservation properties as the conforming reference mesh.

\begin{figure}
  \centering
  \begin{minipage}{0.4\textwidth}
  \includegraphics{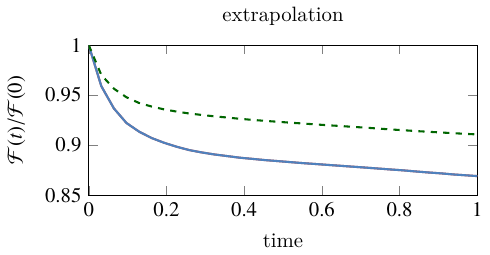}
  \end{minipage}
  \hfill
  \begin{minipage}{0.4\textwidth}
  \includegraphics{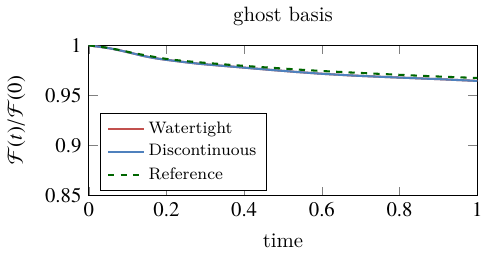}
  \end{minipage}
  \caption{Normalized energy versus time using a pseudorandom initial condition
  with mesh given in Figure~\ref{fig:const:mesh} for $n=7$.\label{fig:energy}}
\end{figure}
We now turn to confirming the energy stability properties of the scheme. To do
this, we use the pseudorandom initial condition with periodic boundary
conditions. In Figure~\ref{fig:energy} normalized energy in the solution versus
time is shown for $n=7$ for both the extrapolation and ghost basis GD methods.
As can be seen, both the schemes with all of the different geometry treatments
dissipate energy in time as predicted by the previous stability analysis
and the dissipation is essentially the same for both approximate geometry
treatments; by
using a pseudorandom initial value we are ensuring that energy is widely spread
across the eigenmodes of the operator.

It is worth noting that the approach we have taken to make the mesh watertight
will not work in general for GD elements for several reasons. First, as the mesh
within an element is refined the geometry approximation does not change.  For
homogeneous materials, as considered here, this is not a problem because the
interface is purely computational (e.g., the accuracy of the solution does not
depend on this interface being represented accurately). For heterogeneous
materials with discontinuous material properties (not considered here) this may
be a problem since convergence will require accurate resolution of the
interface. An additional problem with this approach is that by only moving
points on the interface it is possible that points will be moved past their
neighbors, giving rise to a negative Jacobian determinant; this is in particular
an issue at high resolution.  Lastly, the modification of the geometry can
result in a significant change in the time step restriction. For instance, when
this problem was run with $n=5$ the time step restriction for the watertight
mesh was $6$ times smaller than for the discontinuous mesh.

There are several possible approaches to address these issues if the
constant-preserving and conservation properties are desired. For purely
computational interfaces, a transfinite blending of the surface into the volume
could be used in instead of a global transform. This would mean that as the GD
mesh is refined the element could not become inverted (assuming the transfinite
blend did not invert the element).
For physical interfaces, one could resolve the interface more accurately by
adding elements to the mesh (as opposed to increasing the grid lines within the
elements).

\subsection{Accuracy of WADG-GD on Nonconforming Meshes}
In this test, we explore the accuracy of WAGD-DG\@. Using the
previously defined transform~\eref{eqn:global:tran} (see
Figure~\ref{fig:conservation}) and with the zero velocity boundary
condition~\eref{eqn:bc}. For any $k \in \mathbb{Z}$ the modal solution
\begin{align}
  \label{eqn:modal:box}
  p   &= \sqrt{2}           \cos\left(\pi k \frac{x+1}{2}\right) \cos\left(\pi k \frac{y+1}{2}\right) \cos\left(\frac{\sqrt{2}}{2} \pi k t\right),\\
  v_x &= \phantom{\sqrt{2}} \sin\left(\pi k \frac{x+1}{2}\right) \cos\left(\pi k \frac{y+1}{2}\right) \sin\left(\frac{\sqrt{2}}{2} \pi k t\right),\\
  v_y &= \phantom{\sqrt{2}} \cos\left(\pi k \frac{x+1}{2}\right) \sin\left(\pi k \frac{y+1}{2}\right) \sin\left(\frac{\sqrt{2}}{2} \pi k_{n} t\right),
\end{align}
satisfies governing equations~\eref{eqn:strong} with~\eref{eqn:bc} on $\Omega =
[-1,1] \times [-1,1]$. Using $k = 15$ we run this test for both the
extrapolation and ghost basis GD schemes with varying orders and mesh
resolutions. The final time for the simulation is $t = 2 \sqrt{2} / k$ which
corresponds to one full oscillation of the solution. The base mesh for all
orders is as shown in the figure; the bottom-right and top-left elements use an
$16\times16$ (interior) GD mesh and a $21\times21$ (interior) GD mesh for the
other two elements. We refine the mesh by doubling the resolution inside each
element. For this test we only consider the discontinuous geometry
representation, with an (approximate) $L^{2}$-projection of the coordinate
transform into the GD space.

Results for the test are shown in Figure~\ref{fig:curved:GDbox} and
Table~\ref{tab:curved:GDbox}. As can be seen, the method is converging
at a
high-order rate for each $n$ as the resolution is increased. For the same grid
spacing, the ghost basis method outperforms the extrapolation method.
\begin{figure}
  \centering
  \includegraphics{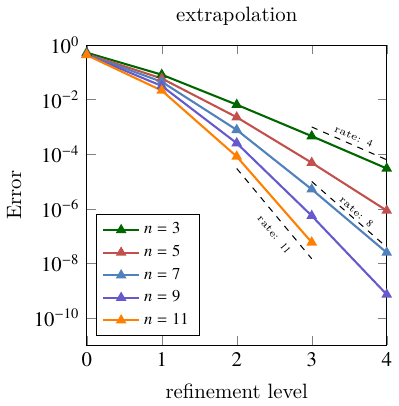}
  \includegraphics{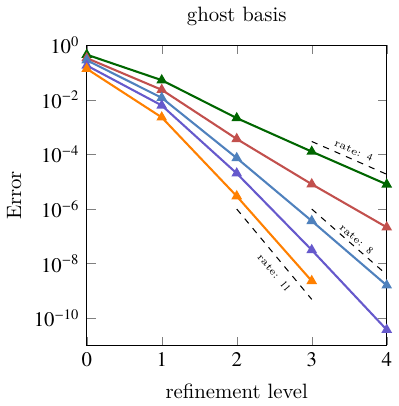}
  \caption{Error plot for extrapolation and ghost basis WADG-GD with modal
  solution~\eref{eqn:modal:box} using coordinate
  transform~\eref{eqn:global:tran}.\label{fig:curved:GDbox}}
\end{figure}
\begin{table}
  \centering
  \sisetup{
    round-mode = places,
    round-precision = 2
  }  \begin{tabular}{r*{5}{l}}
    \toprule
    & \multicolumn{1}{c}{$n = 3$} & \multicolumn{1}{c}{$n = 5$}
    & \multicolumn{1}{c}{$n = 7$} & \multicolumn{1}{c}{$n = 9$}
    & \multicolumn{1}{c}{$n = 11$}\\
    level & \multicolumn{1}{c}{error (rate)} & \multicolumn{1}{c}{error (rate)} & \multicolumn{1}{c}{error (rate)} & \multicolumn{1}{c}{error (rate)} & \multicolumn{1}{c}{error (rate)}\\
    \midrule
    \multicolumn{6}{c}{WADG-GD (extrapolation)}\\
     $0$ & \num{5.278241531733167e-01} \hfill                               & \num{4.680890313445866e-01} \hfill                               & \num{4.501327889709894e-01} \hfill                               & \num{4.442725248433605e-01} \hfill                               & \num{4.312916176453430e-01} \hfill                              \\
     $1$ & \num{8.267765157852183e-02} \hfill (\num{2.674488054381229e+00}) & \num{6.004596142297955e-02} \hfill (\num{2.962643837618395e+00}) & \num{4.533602104661039e-02} \hfill (\num{3.311620976196402e+00}) & \num{3.189668665984739e-02} \hfill (\num{3.799966449792017e+00}) & \num{2.162472706849611e-02} \hfill (\num{4.317909847556565e+00})\\
     $2$ & \num{6.547878160925849e-03} \hfill (\num{3.658398027890305e+00}) & \num{2.271393513396949e-03} \hfill (\num{4.724417640057879e+00}) & \num{7.801717525817215e-04} \hfill (\num{5.860722202955637e+00}) & \num{2.541482118669331e-04} \hfill (\num{6.971592678027062e+00}) & \num{8.224999593988012e-05} \hfill (\num{8.038450601251835e+00})\\
     $3$ & \num{4.602943799617403e-04} \hfill (\num{3.830398744274308e+00}) & \num{4.867336946488549e-05} \hfill (\num{5.544301210731073e+00}) & \num{5.159809863094978e-06} \hfill (\num{7.240330050005244e+00}) & \num{5.517899974663394e-07} \hfill (\num{8.847335060409781e+00}) & \num{5.846448542835842e-08} \hfill (\num{10.45823937452516e+00})\\
     $4$ & \num{3.005257581651583e-05} \hfill (\num{3.936996272619943e+00}) & \num{8.584235803992650e-07} \hfill (\num{5.825299133283498e+00}) & \num{2.471727473927287e-08} \hfill (\num{7.705654409367351e+00}) & \num{7.195263460917866e-10} \hfill (\num{9.582856075755167e+00}) &                             \hfill                              \\
    \midrule
    \multicolumn{6}{c}{WADG-GD (ghost basis)}\\
     $0$ & \num{4.571557063635704e-01} \hfill                               & \num{3.419869827594216e-01} \hfill                               & \num{2.875400607334112e-01} \hfill                               & \num{1.833503267356858e-01} \hfill                               & \num{1.386917765485475e-01} \hfill                              \\
     $1$ & \num{5.301316115032918e-02} \hfill (\num{3.108263152641414e+00}) & \num{2.351885764031444e-02} \hfill (\num{3.862051519754254e+00}) & \num{1.193482505707691e-02} \hfill (\num{4.590513645214815e+00}) & \num{6.280086058936816e-03} \hfill (\num{4.867674697578933e+00}) & \num{2.291395903410760e-03} \hfill (\num{5.919511691729250e+00})\\
     $2$ & \num{2.155280152967059e-03} \hfill (\num{4.620403255836337e+00}) & \num{3.711422723699117e-04} \hfill (\num{5.985701846458552e+00}) & \num{7.385725699443476e-05} \hfill (\num{7.336222021060834e+00}) & \num{2.028074583577549e-05} \hfill (\num{8.274529809713274e+00}) & \num{2.954996862419789e-06} \hfill (\num{9.598854432621158e+00})\\
     $3$ & \num{1.274743204786010e-04} \hfill (\num{4.079596856873435e+00}) & \num{8.138866450801820e-06} \hfill (\num{5.511000644665891e+00}) & \num{3.658374426689140e-07} \hfill (\num{7.657393133301697e+00}) & \num{3.134027948667796e-08} \hfill (\num{9.337876948347230e+00}) & \num{2.300947954896926e-09} \hfill (\num{10.32671253159323e+00})\\
     $4$ & \num{7.993286275235284e-06} \hfill (\num{3.995274077749775e+00}) & \num{2.138588575580875e-07} \hfill (\num{5.250097008828048e+00}) & \num{1.585865712691288e-09} \hfill (\num{7.849788316872291e+00}) & \num{3.738919384305316e-11} \hfill (\num{9.711180964956956e+00}) &                             \hfill                              \\
    \bottomrule
  \end{tabular}
  \caption{Error and calculated rates for WADG-GD with modal
  solution~\eref{eqn:modal:box} using coordinate
  transform~\eref{eqn:global:tran}.\label{tab:curved:GDbox}}
\end{table}

\subsection{Simplicial elements vs GD elements comparison}\label{sec:dg:gd:comp}
We now turn to the coupling of simplicial and GD elements.  In order
to motivate this coupling, in this section we seek to compare the time step
restrictions and number of degrees of freedom needed for the two classes of
schemes. The challenge in doing this is the selection of the mesh resolution
for each method.  Just choosing the two meshes so that they have the same
number of degrees of freedom may not be the right choice, since even though the
methods may converge at the same rate the constant multiplying the error terms
may be different. We address this by setting the resolutions for each method so
that the error for each method is roughly the same for a modal problem and
then determine the maximum stable time step for each method. In this test, the
simplicial mesh is taken to be a mesh of $64$ right triangles; the mesh is
generated by initially splitting the square into 4 right triangles and then
iterative refining the mesh 2 times with quadrisection refinement.

\begin{table}
  \centering
  \sisetup{
    round-mode = places,
    round-precision = 1
  }  \begin{tabular}{rr|cccc|cccc}
    \toprule
    &
    & \multicolumn{4}{c|}{extrapolation}
    & \multicolumn{4}{c}{ghost basis}\\
    $n$ & $k_{n}$ &
    $N^{E}$ & $\displaystyle\frac{N_{p}^{\mathbb{P} }}{N_{p}^{E}}$ & $\displaystyle\frac{\Delta t^{E}}{\Delta t^{\mathbb{P} }}$ & $\displaystyle\frac{\Delta t^{E}N_{p}^{\mathbb{P} }}{N_{p}^{E}\Delta t^{\mathbb{P} }}$ &
    $N^{G}$ & $\displaystyle\frac{N_{p}^{\mathbb{P} }}{N_{p}^{G}}$ & $\displaystyle\frac{\Delta t^{G}}{\Delta t^{\mathbb{P} }}$ & $\displaystyle\frac{\Delta t^{G}N_{p}^{\mathbb{P} }}{N_{p}^{G}\Delta t^{\mathbb{P} }}$\\
    \midrule
     $3$ &  $1$ & $15$ & \num{2.5000e+00} & \num{2.2340e+00} & \num{5.5849e+00} & $11$ & \num{3.2653e+00} & \num{1.6567e+00} & \num{5.4097e+00}\\
     $5$ &  $4$ & $27$ & \num{1.7143e+00} & \num{1.6996e+00} & \num{2.9137e+00} & $24$ & \num{1.5981e+00} & \num{1.3436e+00} & \num{2.1472e+00}\\
     $7$ &  $7$ & $30$ & \num{2.3975e+00} & \num{1.9286e+00} & \num{4.6239e+00} & $23$ & \num{2.5600e+00} & \num{1.3616e+00} & \num{3.4857e+00}\\
     $9$ & $12$ & $38$ & \num{2.3143e+00} & \num{1.8428e+00} & \num{4.2647e+00} & $30$ & \num{2.3143e+00} & \num{1.3436e+00} & \num{3.1096e+00}\\
    $11$ & $15$ & $43$ & \num{2.5785e+00} & \num{1.9091e+00} & \num{4.9228e+00} & $34$ & \num{2.4652e+00} & \num{1.5698e+00} & \num{3.8699e+00}\\
    \bottomrule
  \end{tabular}
  \caption{Polynomial discontinuous Galerkin and Galerkin Difference time step
  comparison\label{tab:dt}}
\end{table}
In this test we take the domain to be the unit square $\Omega = [-1,1] \times
[-1,1]$ with the modal solution~\eref{eqn:modal:box}. The mode number $k_{n}$
used for each polynomial order $n$ is given in Table~\ref{tab:dt}. The mode
number is selected to be the largest value $k_{n}$ that results in an error
smaller that $10^{-3}$ at time $t = 2 \sqrt{2} / k_{n}$ using simplicial
elements with polynomial order $n$; as noted above the simplicial element mesh
is fixed at $64$ elements.

For the GD method, the unit square is discretized using a single GD block that has
$(N+1) \times (N+1)$ interior GD points; the total number of points depends on
whether the extrapolation or ghost basis method is used. For each polynomial
order $n$, the value of $N$ is chosen to be the smallest value so that the error
at time $t = 2\sqrt{2} / k_{n}$ is less than the simplicial element solution
with the same polynomial order.  Table~\ref{tab:dt} gives the value of $N$
used for GD with the extrapolation and ghost basis boundary treatments; in the
table superscript $E$ is used for values related to the extrapolation GD
method, subscript $G$ is used for the ghost basis GD method, and $\mathbb{P}$
for simplicial elements.

A first observation is that for the same grid spacing
the ghost basis method is more accurate than the extrapolation method, namely
$2/N^{G} > 2/N^{E}$. That said, in some cases ($n=5$ and $n=11$) the ghost basis
method requires more degrees of freedom; for a given grid size the total degrees
of freedom for the extrapolation method is $N_{p}^{E} = {(N^{E}+1)}^2$ as
compared with $N_{p}^{G} = {(N^{G}+n)}^2$ for the ghost basis method. That said,
both GD element types out perform simplicial elements in terms of number of
degrees of freedom; the $64$ simplicial element mesh requires
$N_{p}^{\mathbb{P}} = 64 (n+1)(n+2)/2$ degrees of freedom. This is seen columns
labeled $N_p^{\mathbb{P}}/N_{p}^{E}$ and $N_p^{\mathbb{P}}/N_{p}^{G}$ of
Table~\ref{tab:dt} where the ratio total degrees of freedom is given.
Since all these ratios are greater than $1$, both GD element types are
outperforming the simplicial elements with these parameters. This
comparison is important because if the schemes are assumed to be memory
bandwidth limited on current computing architectures, the total number of
degrees of freedom can be taken as a proxy for the time to perform a single
right-hand side evaluation.

We now turn to the consideration of the maximum stable time step for each
method. We determine the maximum time step by initializing the solution vectors
to pseudorandom values, then taking $100$ time steps. The maximum time step is
then the largest time step that results in no energy growth; by using a
pseudorandom initial condition energy is spread across the eigenmodes of the
numerical scheme. Table~\ref{tab:dt} gives the ratio of
the maximum time step of the two GD methods ($\Delta t^{E}$ and $\Delta t^{G}$)
to the simplicial element method ($\Delta t^\mathbb{P}$); note that when
comparing time steps the GD time step is on top and when comparing number of
degrees of freedom DG was on top, thus in both cases numbers larger than $1$
favor GD\@. As can be seen, for the given mesh resolutions, meshes with either
of the GD element types can take larger time steps than meshes with simplicial
elements\@. We also note that even though the ghost basis has more favorable
accuracy for a given grid spacing, the time step is smaller.

\begin{table}
  \centering
  \sisetup{
    round-mode = places,
    round-precision = 2
  }  \begin{tabular}{rccccc}
    \toprule
    & simplicial
    & extrapolation
    & ghost basis\\
    $n$
    & $\displaystyle\frac{n\Delta t^{\mathbb{P} }}{2r^{\mathbb{P} }}$
    & $\displaystyle\frac{\Delta t^{E}}{h^{E}}$
    & $\displaystyle\frac{\Delta t^{G}}{h^{G}}$
    & $m$
    & $\rho$
    \\
    \midrule
     $3$ & \num{0.30995} & \num{0.50702} & \num{0.27574} & 4  & \num{-2.7853e+00}\\
     $5$ & \num{0.38457} & \num{0.51690} & \num{0.36322} & 7  & \num{-3.9541e+00}\\
     $7$ & \num{0.36313} & \num{0.43956} & \num{0.23791} & 8  & \num{-4.3136e+00}\\
     $9$ & \num{0.39834} & \num{0.45389} & \num{0.26127} & 11 & \num{-5.4504e+00}\\
    $11$ & \num{0.37513} & \num{0.41000} & \num{0.26656} & 12 & \num{-5.8228e+00}\\
    \bottomrule
  \end{tabular}
  \caption{CFL restriction comparison for simplicial and GD elements;
  see also Table~\ref{tab:dt}. Here $m$ is the order of the Taylor time stepping
  method that has been used and $\rho$ minimum extent of the stability
  region.\label{tab:cfl}}
\end{table}
An important observation for the time step restriction is that the ratio is
roughly constant across polynomial orders once the error level for the schemes
has been fixed. To explore this further, in Table~\ref{tab:cfl} we give the CFL
restrictions for each of the methods and polynomial orders. For
simplicial elements, the time step is expected to scale with $2r^\mathbb{P} / n$
where $r^{\mathbb{P}}$ is the radius of the inscribe circle for the elements; on
this mesh $r^{\mathbb{P}} = \sqrt{2} / 4(1+\sqrt{2})$. For the GD elements the
time step should scale with mesh size, namely $h^{E} = 2/N^{E}$ and $h^{G} =
2/N^{G}$.
As can be seen in Table~\ref{tab:cfl}, the time step restriction for the Taylor
time integration method is roughly constant with respect to polynomial order for
all three element types, and the extrapolation GD is favorable to the ghost
basis GD by a factor of 1.5 to 2.
We note that the stability region increases with order for the Taylor time
integration method, so the extent of the eigenvalue spectra do increase with
order.
This can be seen in the last two columns of Table~\ref{tab:cfl} where
order of the Taylor time integration method for spatial order
$n$ and the minimum real extent of the stability region ($\rho$) is given.

\begin{figure}
  \centering
  \includegraphics{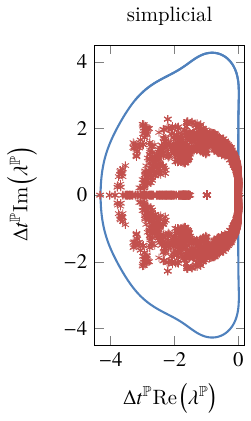}
  \hfill
  \includegraphics{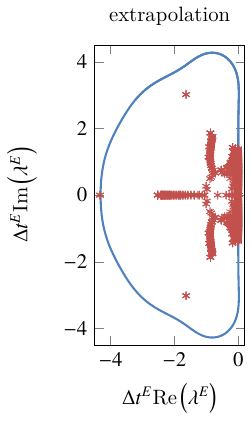}
  \hfill
  \includegraphics{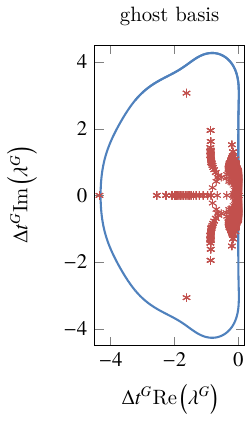}
  \caption{Time step scaled eigenvalue spectra for the methods with $n=7$ along
  with the order $8$ Taylor time integration stability region; see
  Table~\ref{tab:dt}.\label{fig:eig}}
\end{figure}
To explore this further, Figure~\ref{fig:eig} shows the time step scaled
eigenvalue spectra for the three methods with $n=7$ along with the order $8$
Taylor time integration scheme stability region; the eigenvalues and stability
region have been computed numerically. As can be seen, it is a purely real
eigenvalue that limits the time step for each method. For the GD schemes, the
eigenmode (not shown) associated with the time step restrictive eigenvalues has
most of its energy at degrees of freedom at the corners of the element. We note
that this phenomenon is multidimensional; in one space dimension the spectral
radius of the GD differentiation matrices was found to be bounded independent of
polynomial degree~\cite{BanksHagstrom2016JCP}. It is surprising that for a
tensor product grid this property is not preserved, so the origin and,
hopefully, suppression of the corner modes is a topic for future study.

Since the inverse of the time step is the number of steps required to complete a
simulation, if the schemes are memory bandwidth limited the overall time to solution
savings (as compared with the simplicial element method is) is $(\Delta t^{E}/N_{p}^{E})
(N_{p}^{\mathbb{P}}/\Delta t^{\mathbb{P}})$ for extrapolation GD and $(\Delta
t^{G}/N_{p}^{G}) (N_{p}^{\mathbb{P}}/\Delta t^{\mathbb{P}})$ for ghost basis
GD\@.  As can be seen in Table~\ref{tab:dt}, both GD schemes are favorable in
this metric. Additionally, in all cases except $n = 3$, the extrapolation method
out performs the ghost basis method.

\subsection{Waves on a Disk}
In this test we take $\Omega$ to the unit disk and consider the analytic
solution
\begin{align}
  p = \frac{\partial^2 u}{\partial t^2},&\quad
  v_{x} = \frac{\partial^2 u}{\partial t\partial x},\quad
  v_{y} = \frac{\partial^2 u}{\partial t\partial y},
\end{align}
with the function $u$ being
\begin{align}
  u &= \alpha \cos\left(R_{0} t - \beta \theta\right)J_{\beta}\left(R_{0} r\right).
\end{align}
Here $J_{\beta}(r)$ is the Bessel function of the first kind and $(r, \theta)$
are the polar coordinates of the Cartesian point $(x,y)$. In the test we use
$\beta = 7$, choose $R_{0}$ so that $J'_{\beta}(R_{0}) = 0$ which ensures that
the $\vec{n}\cdot\vec{v} = 0$ on the outer edge of the disk, and $\alpha$
is a normalization constant so that the energy in the solution is $1$. For
$n=3$, $5$, and $7$ we choose $R_{0}$ to be the $2^{nd}$ root so that $R_{0}
\approx 12.93$ with $\alpha \approx 0.024$, and for $n=9$ and $11$ we choose
$R_{0}$ to be the $5^{th}$ root so that $R_{0} \approx 109.6$ with
$\alpha\approx0.0091$; both $R_{0}$ and $\alpha$ are determined numerically.

Here we consider three different methods for handling the disk geometry:
\begin{enumerate}
  \item WADG on simplicial elements: base mesh of 48 simplicial elements,
  \item WADG on simplicial elements coupled with GD\@: base mesh of 20
    simplicial elements coupled with a central GD element in the center (see
    right panel of Figure~\ref{fig:decomp}), and
  \item WADG-GD\@: 5 GD elements with the 4 edge elements being curved (see left
    panel of Figure~\ref{fig:decomp}).
\end{enumerate}
In the case of the purely simplicial mesh, the 20 elements around the edge of
the base mesh match those in the right panel of Figure~\ref{fig:decomp}, e.g.,
the GD element is replaced with 28 simplicial elements.

For the coupled simplicial GD problem, the GD element has lower-left and
upper-right corners $(\pm \sqrt{2}/2,\pm \sqrt{2}/2)$ and is discretized using
an $(N_{0}+1) \times (N_{0}+1)$ grid of points, where $N_{0}$ for the
extrapolation and ghost basis methods are given in Table~\ref{tab:disk:param}.
These parameters were chosen so that the coupled method has a lower error than
the purely simplicial method at time $t = 2 \pi \beta/ R_{0}$ with the same $n$
for refinement level 0 (with the exception of $n=3$ where refinement level 1 is
considered).

For the WADG-GD mesh, the center GD element has lower-left and
upper-right corners $(\pm 1/3, \pm 1/3)$ and is discretized using an
$(N_{1}+1)\times(N_{1}+1)$ interior grid. The outer elements have interior grids
that are $(N_{2}+1)\times(N_{3}+1)$ where $N_{2}$ and $N_{3}$ are chosen so that
the grid spacing along edges of the element are no larger than the grid spacing
in the center element. The values for the extrapolation and ghost basis method
used are given in Table~\ref{tab:disk:param}, where the resolution is chosen in
the same manner as for the coupled simplicial-GD element mesh.

\begin{table}
  \centering
  \sisetup{
    round-mode = places,
    round-precision = 1
  }  \begin{tabular}{r|cc|cc}
    \toprule
    & \multicolumn{2}{c|}{GD coupled to WADG simplicial}
    & \multicolumn{2}{c}{WADG-GD}\\
    & \multicolumn{1}{c}{extrapolation}
    & \multicolumn{1}{c|}{ghost basis}
    & \multicolumn{1}{c}{extrapolation}
    & \multicolumn{1}{c}{ghost basis}\\
    $n$
    & $N_{0}^{E}$
    & $N_{0}^{G}$
    & $N_{1}^{E}$, $N_{2}^{E}$, $N_{3}^{E}$
    & $N_{1}^{G}$, $N_{2}^{G}$, $N_{3}^{G}$
    \\
    \midrule
     3 & 11 &  9 &  4,  6, 10 &  4,  6, 10\\
     5 & 18 & 14 &  7, 10, 17 &  6,  9, 15\\
     7 & 25 & 19 & 10, 15, 24 &  7, 10, 17\\
     9 & 31 & 24 & 14, 20, 33 & 11, 16, 26\\
    11 & 37 & 28 & 16, 23, 38 & 13, 19, 31\\
    \bottomrule
  \end{tabular}
  \caption{GD parameters used for the disk test problem.\label{tab:disk:param}}
\end{table}

\begin{table}
  \centering
  \sisetup{
    round-mode = places,
    round-precision = 2
  }  \begin{tabular}{r*{3}{c}}
    \toprule
    & \multicolumn{1}{c}{$n = 3$} & \multicolumn{1}{c}{$n = 5$}
    & \multicolumn{1}{c}{$n = 7$}\\
    level & \multicolumn{1}{c}{error (rate)} & \multicolumn{1}{c}{error (rate)} & \multicolumn{1}{c}{error (rate)}\\
    \midrule
    \multicolumn{4}{c}{WADG simplicial elements}\\
    0 & \num{3.611296949496219e-01}                               & \num{6.083116285581943e-03}                               & \num{1.825920799775534e-04}                              \\
    1 & \num{9.203374140358527e-03} (\num{5.294210365874241e+00}) & \num{1.014107533963521e-04} (\num{5.906528037102187e+00}) & \num{8.355251883080034e-07} (\num{7.771725152321618e+00})\\
    2 & \num{4.610489464405055e-04} (\num{4.319171053852672e+00}) & \num{1.653421046924065e-06} (\num{5.938612674217500e+00}) & \num{3.414220046022262e-09} (\num{7.934983468314990e+00})\\
    3 & \num{2.894431242566665e-05} (\num{3.993568129178997e+00}) & \num{2.614815287845413e-08} (\num{5.982601309112714e+00}) & \num{1.367892194823012e-11} (\num{7.963457698410362e+00})\\
    \midrule
    \multicolumn{4}{c}{WADG simplicial with extrapolation GD}\\
    0 & \num{2.478470804730398e-01}                               & \num{5.627389183672092e-03}                               & \num{1.530783757387598e-04}                              \\
    1 & \num{7.785155930010187e-03} (\num{4.992580518844761e+00}) & \num{9.332673178337914e-05} (\num{5.914031557106949e+00}) & \num{6.862176585291379e-07} (\num{7.801388531494359e+00})\\
    2 & \num{4.795183026203974e-04} (\num{4.021068148248689e+00}) & \num{1.537259729154147e-06} (\num{5.923857531995649e+00}) & \num{2.820188314548592e-09} (\num{7.926730941210903e+00})\\
    3 & \num{3.029901176156319e-05} (\num{3.984243237821968e+00}) & \num{2.461084102688672e-08} (\num{5.964923167394024e+00}) & \num{2.377059320271666e-11} (\num{6.890469783129403e+00})\\
    \midrule
    \multicolumn{4}{c}{WADG simplicial with ghost basis GD}\\
    0 & \num{6.500838663651717e-01}                               & \num{6.053000296565362e-03}                               & \num{1.549481285954082e-04}                              \\
    1 & \num{7.654722695344096e-03} (\num{6.408131925170787e+00}) & \num{1.189553156892957e-04} (\num{5.669158774514485e+00}) & \num{6.838147899155189e-07} (\num{7.823963988873074e+00})\\
    2 & \num{4.076192417615750e-04} (\num{4.231056053740304e+00}) & \num{2.444365189363285e-06} (\num{5.604816091271417e+00}) & \num{2.825346579055799e-09} (\num{7.919033965657567e+00})\\
    3 & \num{2.541562604147590e-05} (\num{4.003434482852548e+00}) & \num{5.715512213567377e-08} (\num{5.418433234505450e+00}) & \num{1.636937480544905e-11} (\num{7.431284818288966e+00})\\
    \midrule
    \multicolumn{4}{c}{WADG extrapolation GD}\\
    0 & \num{1.047995203178337e-01}                               & \num{3.401510409018690e-03}                               & \num{1.158588814343971e-04}                              \\
    1 & \num{6.538621829227872e-03} (\num{4.002501717930536e+00}) & \num{8.590376095223496e-05} (\num{5.307310398333870e+00}) & \num{5.890927376318919e-07} (\num{7.619658158792015e+00})\\
    2 & \num{4.932859763617259e-04} (\num{3.728490406698442e+00}) & \num{1.762604570249927e-06} (\num{5.606940539270076e+00}) & \num{2.880278963514050e-09} (\num{7.676142409554920e+00})\\
    3 & \num{3.403450157813931e-05} (\num{3.857354386206650e+00}) & \num{3.136477088595773e-08} (\num{5.812420016706740e+00}) & \num{1.246326253519226e-11} (\num{7.852382962277724e+00})\\
    \midrule
    \multicolumn{4}{c}{WADG ghost basis GD}\\
    0 & \num{5.995894486004658e-02}                               & \num{3.369897179598947e-03}                               & \num{1.762225861330997e-04}                              \\
    1 & \num{2.207675177386475e-03} (\num{4.763375171075338e+00}) & \num{1.815592990860356e-04} (\num{4.214191844428171e+00}) & \num{1.064421372522312e-06} (\num{7.371185649877119e+00})\\
    2 & \num{1.401578591911452e-04} (\num{3.977403369975927e+00}) & \num{3.399257777855724e-06} (\num{5.739077242870874e+00}) & \num{5.095931116675839e-09} (\num{7.706507891949374e+00})\\
    3 & \num{9.220026503347084e-06} (\num{3.926137936135417e+00}) & \num{7.748802124029840e-08} (\num{5.455102657009050e+00}) & \num{2.187180012217608e-11} (\num{7.864130002240985e+00})\\
    \bottomrule
  \end{tabular}
  \caption{Disk results for $R_{0} \approx 12.93$ and $\alpha \approx
  0.024$.\label{tab:low:disk}}
\end{table}
\begin{table}
  \centering
  \sisetup{
    round-mode = places,
    round-precision = 2
  }  \begin{tabular}{r*{3}{c}}
    \toprule
    & \multicolumn{1}{c}{$n = 9$} & \multicolumn{1}{c}{$n = 11$}\\
    level & \multicolumn{1}{c}{error (rate)} & \multicolumn{1}{c}{error (rate)}\\
    \midrule
    \multicolumn{3}{c}{WADG simplicial elements}\\
    0 & \num{8.861561266116692e-04}                               & \num{4.149992234824648e-05}                              \\
    1 & \num{1.217035639173385e-06} (\num{9.508045674902130e+00}) & \num{1.412816408540612e-08} (\num{11.52031891790478e+00})\\
    2 & \num{1.316541602170609e-09} (\num{9.852402589951810e+00}) & \num{3.855427884159626e-12} (\num{11.83939540069439e+00})\\
    \midrule
    \multicolumn{3}{c}{WADG simplicial with extrapolation GD}\\
    0 & \num{7.099699663471176e-04}                               & \num{3.491643268586258e-05}                              \\
    1 & \num{9.570921920454426e-07} (\num{9.534884381446060e+00}) & \num{1.138750296098981e-08} (\num{11.58223902588830e+00})\\
    2 & \num{1.061048750220075e-09} (\num{9.817023146135465e+00}) & \num{7.065335946341717e-12} (\num{10.65440564848854e+00})\\
    \midrule
    \multicolumn{3}{c}{WADG simplicial with ghost basis GD}\\
    0 & \num{8.578966229555754e-04}                               & \num{3.957039508530920e-05}                              \\
    1 & \num{9.714387393393707e-07} (\num{9.786465077509680e+00}) & \num{1.607328653272746e-08} (\num{11.26554080491132e+00})\\
    2 & \num{1.052588914601808e-09} (\num{9.850037103328855e+00}) & \num{5.631280087747494e-11} (\num{8.156986324971022e+00})\\
    \midrule
    \multicolumn{3}{c}{WADG extrapolation GD}\\
    0 & \num{4.702187427996997e-04}                               & \num{3.569955628973031e-05}                              \\
    1 & \num{8.776283657325478e-07} (\num{9.065506174570364e+00}) & \num{1.847951954552407e-08} (\num{10.91576317956372e+00})\\
    2 & \num{1.124975162869076e-09} (\num{9.607573195398022e+00}) & \num{1.094083674058746e-11} (\num{10.72198845510951e+00})\\
    \midrule
    \multicolumn{3}{c}{WADG ghost basis GD}\\
    0 & \num{6.294842476601364e-04}                               & \num{4.232843605625833e-05}                              \\
    1 & \num{6.792468851083185e-07} (\num{9.856018516219788e+00}) & \num{1.187863328043767e-08} (\num{11.79904261629000e+00})\\
    2 & \num{8.218839677536524e-10} (\num{9.690785598275131e+00}) & \num{1.531812390975664e-10} (\num{6.276985429617779e+00})\\
    \bottomrule
  \end{tabular}
  \caption{Disk results for $R_{0} \approx 109.6$ and
  $\alpha\approx0.0091$.\label{tab:high:disk}}
\end{table}
Results for the disk test problem are given in Table~\ref{tab:low:disk} for
$n=3, 5, 7$ and in Table~\ref{tab:high:disk} for $n=9, 11$. As can be seen all
of the schemes perform as expected converging at high order.

\begin{table}
  \centering
  \sisetup{
    round-mode = places,
    round-precision = 1
  }  \begin{tabular}{rccccc}
    \toprule
    &
    & Simplicial with
    & Simplicial with
    \\
    $n$
    & Simplicial
    & extrapolation GD
    & ghost basis GD
    & extrapolation GD
    & ghost basis GD
    \\
    \midrule
      3 & \num{3.381557709571431e-02} & \num{3.857699388980444e-02} & \num{3.864228073749922e-02} & \num{2.630305358962537e-02} & \num{1.319501614859978e-02}\\
      5 & \num{2.570868318934896e-02} & \num{2.816630801157351e-02} & \num{2.814419150007283e-02} & \num{1.517996900388158e-02} & \num{1.188810059965932e-02}\\
      7 & \num{1.762932196002077e-02} & \num{1.940242771362551e-02} & \num{1.770238337686799e-02} & \num{8.694612820182083e-03} & \num{7.027439380337052e-03}\\
      9 & \num{1.485950929183757e-02} & \num{1.619367473248401e-02} & \num{1.501419514002846e-02} & \num{6.534411302347341e-03} & \num{4.652811135037450e-03}\\
     11 & \num{1.127438061395471e-02} & \num{1.226493022250229e-02} & \num{1.187971648584490e-02} & \num{5.093361513156943e-03} & \num{3.987999005999079e-03}\\
    \bottomrule
  \end{tabular}
  \caption{Maximum stable time step for each of the schemes base mesh for the
  disk using the parameters in Table~\ref{tab:disk:param}\label{tab:disk:dt}}
\end{table}
An important question is whether the coupling of GD with simplicial elements has
negatively affected the time step of the scheme. In Table~\ref{tab:disk:dt} we
give the maximum stable time step (determined as in
Section~\ref{sec:dg:gd:comp}) for each of the schemes. As can be seen, the
coupled simplicial and GD schemes have similar time step restrictions as the
pure simplicial method. Additionally, the benefits of handling geometry with
smaller simplicial elements is clear from the increased time step as compared
with the purely GD schemes; though we note that using more GD elements would
likely improve the GD time step since the grid spacing could be made more
uniform.

\subsection{Inclusion Scattering}
\begin{figure}
  \centering
  \includegraphics{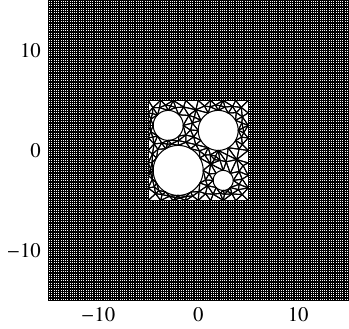}
  \includegraphics{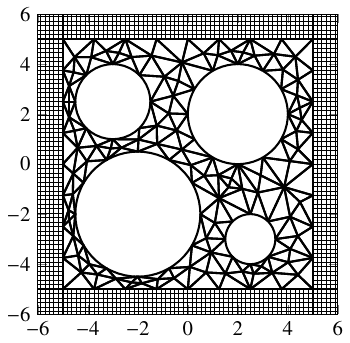}
  \caption{Computational mesh used for the inclusion test problem. (left) Full
  computational domain and base mesh. (right) Close up view showing the base
  simplicial mesh along with the GD-simplicial
  interface.\label{fig:inclusion:mesh}}
\end{figure}
As a final test we consider the scattering off of four cavity cylinders in an
acoustic medium. The domain is of size $15 \times 15$ and the four cylinder
centers $(x,y)$ and radii $R$ are
\begin{align}
  x_{1} &= \phantom{-}2  , & y_{1} &= \phantom{-}2  , & R_{1} &= 2  , &
  x_{2} &=           -2  , & y_{2} &=           -2  , & R_{2} &= 2.5, \\
  x_{3} &=           -3  , & y_{3} &= \phantom{-}2.5, & R_{3} &= 1.5, &
  x_{4} &= \phantom{-}2.5, & y_{4} &=           -3  , & R_{4} &= 1  .
\end{align}
The computational mesh is shown in
Figure~\ref{fig:inclusion:mesh}.  Around the four cylinders $240$ simplicial
elements are used and away from the cylinders $8$ large GD elements are used;
each of the GD elements is $5 \times 5$ and the region covered by
simplicial elements is of size $5 \times 5$. On the base mesh, each of the
GD-simplicial interfaces has 8 simplicial elements and each of the GD elements
uses an $(8n+1)\times(8n+1)$ interior grid; the GD grid spacing is thus $h =
5/8n$ in each direction. The outer and cylindrical inclusion boundaries are
taken to satisfy the zero velocity boundary condition~\eqref{eqn:bc}. For this
test we consider polynomial orders $n=5$ and $n=7$ with GD elements using the
extrapolation basis.

The initial condition is taken to be
\begin{align}
  p_{0} = \exp\left(-2{(r-10)}^2\right),
  \qquad
  \vec{v}_{0} = \vec{0},
\end{align}
where $r = \sqrt{x^2 + y^2}$. This initial condition is a ring of radius
$8$ centered at the origin. The pressure field at times $t=5$, 10, 15, and 30 is
shown in Figure~\ref{fig:inclusion:field}. As can be seen the wave field becomes
quite complex through the continued interaction with the cylinders.
\begin{figure}
  \centering
  \includegraphics{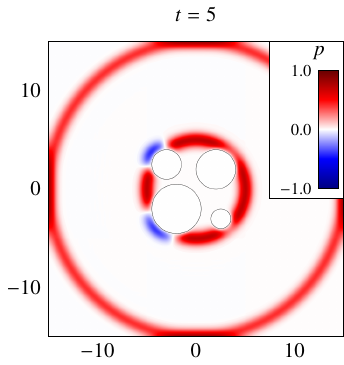}
  \includegraphics{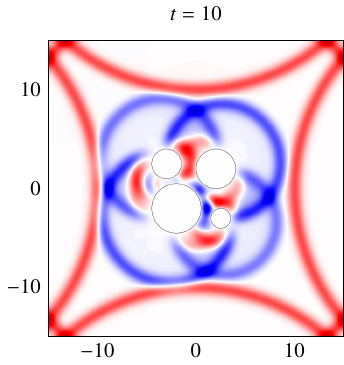}
  \includegraphics{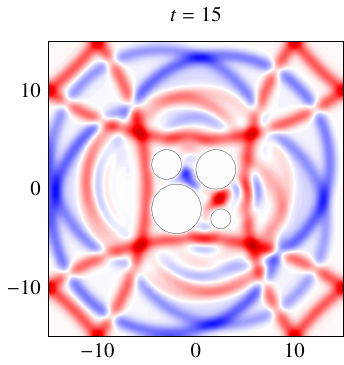}
  \includegraphics{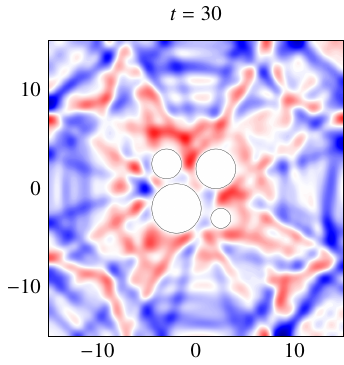}

  \caption{Computed pressure field solution for the inclusion problem. Solution
  is shown for $n=5$ on the base mesh. The same saturated $-1$ to $1$ color
  scale is used for all panels with the colorbar shown in the top-left
  panel.\label{fig:inclusion:field}}
\end{figure}

Since we do not know the analytic solution for this problem, we assess the
accuracy of the scheme through a self-convergence study. Namely, for a fixed $n$
we run the solution at three resolutions; the mesh is refined as before using
grid doubling for the GD elements and quadrisection for the simplicial
elements. We denote the solution on level $l$ as $p^{(l)}$, $v_{x}^{(l)}$, and
$v_{y}^{(l)}$, with level $0$ being the base resolution and level $2$ being the
finest resolution. The convergence rate is estimated as
\begin{align}
  \mbox{rate estimate} &=
  \log_{2} \left(\frac{\Delta^{(1)}}{\Delta^{(2)}}\right),\\
  \label{eqn:delta:l}
  {\left(\Delta^{(l)}\right)}^{2}
  &=
  \frac{1}{2}
  \sum_{e\in\EE}
  \int_{\hat{e}} J \left(
  {\left(p^{(l)}-p^{(l-1)}\right)}^{2}
  +
  {\left(v_{x}^{(l)}-v_{x}^{(l-1)}\right)}^{2}
  +
  {\left(v_{y}^{(l)}-v_{y}^{(l-1)}\right)}^{2}\right).
\end{align}
We approximate the integral in \eref{eqn:delta:l} using the finest level's
quadrature rule, namely the solutions from levels $0$ and $1$ are interpolated to
the quadrature nodes of the level $2$ mesh.

When this test is run with polynomial order $n=5$ the estimated convergence rate
is $6.3$ and when $n=7$ the estimated convergence rate is $7.5$ at time $t =
30$.

\section{Conclusions}
In this paper we have shown how the Galerkin difference (GD) method can be used to
efficiently handle complex geometries. Two approaches were considered: the use
of curved GD elements with efficient handling of the mass matrix using the
weight-adjusted approximation and the coupling with simplicial elements. In all
cases the interface between elements can be nonconforming, meaning that
different approximation spaces are used on the two sides of the interface.
Semidiscrete energy stability was achieved through the use of a skew-symmetric
discontinuous Galerkin formulation, with nonconforming interfaces handling using
mortar elements. The scheme was observed to be high-order accurate in all of
the test problems.

When coupling with simplicial elements, the results suggest that if there are
$K$ simplicial elements of order $n$ along an interface the number of GD
degrees of freedom along the interface should be roughly $nK$. We also
observed that for a similar accuracy the GD method in general required fewer
degrees of freedom than polynomial simplicial elements and that the allowed
time step was larger. This suggests, that if implemented efficiently, GD can
lead to a faster time to solution and that the scheme is good candidate for
local time stepping.

One area for future work is the fact that in multiple dimensions, the GD scheme
has a more restrictive time step than in one dimension. The modes restricting
the time step are largely associated with the corners of the elements. This
suggests that by modifying the GD basis near the corners of the element the time
step might be improved.

\section*{Acknowledgments}
Jeremy~E.~Kozdon was partially supported by National Science Foundation Award
EAR-1547596 and Office of Naval Research Award N0001416WX01290. Computational
resources were provided in part by National Science Foundation Award
OAC-1403203.

Thomas~Hagstrom was supported by contracts from the U.S. Department of Energy
ASCR Applied Math Program and NSF Grant DMS-1418871. Any opinions, findings, and
conclusions or recommendations expressed in this material are those of the
authors and do not necessarily reflect the views of the Department of Energy or
the National Science Foundation.

Jeffrey~W.~Banks was supported by contracts from the U.S. Department of Energy
ASCR Applied Math Program, and by a U.S. Presidential Early Career Award for
Scientists and Engineers.

\appendix

\section{Proof of Theorem~\ref{thm:stability}}\label{app:thm:stability}
Taking the time derivative of \eref{eqn:eng:elm} and substituting in
\eref{eqn:semi:p}--\eref{eqn:semi:vy} gives the following element energy
dissipation rate
\begin{align}
  \frac{d\En^{e}}{dt} = \sum_{g\in\GG^{e}} \Delta^{-g}
\end{align}
where the mortar based energy rate of change is
\begin{align}
  \Delta^{-g} =
  -{\left(\vec{p}^{-}\right)}^{T} \mat{S}_{J}^{g} \mat{W}^{g} \vec{v}^{*}_{n}
  +
  {\left(\vec{p}^{-}\right)}^{T} \mat{S}_{J}^{g} \mat{W}^{g} \vec{v}^{-}_{n}
  -
  {\left(\vec{v}_{n}^{-}\right)}^{T} \mat{S}_{J}^{g} \mat{W}^{g} \vec{p}^{*}.
\end{align}
Thus, the energy rate of change depends only on the energy rate of change
across the mortar elements. To complete the proof we only need to consider a
single mortar element $g$ and show that energy does not increase across $g$.
To do this we separately consider the case of $g$ being a boundary mortar
element and a mortar element between two elements.

If $g$ is a boundary mortar element, then substituting in \eref{eqn:flux} for
$\vec{v}^{*}_{n}$ and $\vec{p}^{*}$ with $p^{+} = p^{-}$ and $v_{n}^{+} =
-v_{n}^{-}$ gives
\begin{align}
  \Delta^{-g} = -\alpha {\left(\vec{v}^{-}_{n}\right)}^{T} \mat{S}_{J}^{g}
  \mat{W}^{g} \vec{v}^{-}_{n},
\end{align}
which is non-positive if $\alpha \ge 0$ and $\mat{S}_{J}^{g}$ is positive; note
that $\mat{W}^{g}$ is positive because the mortar quadrature is assumed to be
a positive weight rule. Thus boundary mortars are energy non-increasing.

Recall that the mortar is partitioned such that $g$ is a mortar element between
two elements, $e^{-}$ and $e^{+}$. Since the flux and normal vectors are
described from the viewpoint of an element, the two sides will use equal and
opposite normal vectors. Using the normal of element $e^{-}$ in the calculation
for $e^{+}$ gives
\begin{align}
  \Delta^{+g} =
  {\left(\vec{p}^{+}\right)}^{T} \mat{S}_{J}^{g} \mat{W}^{g} \vec{v}^{*}_{n}
  -
  {\left(\vec{p}^{+}\right)}^{T} \mat{S}_{J}^{g} \mat{W}^{g} \vec{v}^{+}_{n}
  +
  {\left(\vec{v}_{n}^{+}\right)}^{T} \mat{S}_{J}^{g} \mat{W}^{g} \vec{p}^{*}.
\end{align}
With this, the energy rate of change across $g$ is then
\begin{align}
  \Delta^{-g} + \Delta^{+g} =
  {\jmp{\vec{p} }}^{T} \mat{S}_{J}^{g} \mat{W}^{g} \vec{v}^{*}_{n}
  +
  {\jmp{\vec{v}_{n} }}^{T} \mat{S}_{J}^{g} \mat{W}^{g} \vec{p}^{*}
  +
  {\left(\vec{p}^{-}\right)}^{T} \mat{S}_{J}^{g} \mat{W}^{g} \vec{v}^{-}_{n}
  -
  {\left(\vec{p}^{+}\right)}^{T} \mat{S}_{J}^{g} \mat{W}^{g} \vec{v}^{+}_{n}.
\end{align}
Using \eref{eqn:flux} as viewed from $e^{-}$ for $\vec{v}^{*}_{n}$ and
$\vec{p}^{*}$, this becomes
\begin{align}
  \Delta^{-g} + \Delta^{+g} =\;&
  -\frac{\alpha}{2}{\jmp{\vec{p} }}^{T} \mat{S}_{J}^{g} \mat{W}^{g} \jmp{\vec{p}}
  -\frac{\alpha}{2}{\jmp{\vec{v}_{n} }}^{T} \mat{S}_{J}^{g} \mat{W}^{g}
  \jmp{\vec{v}_{n}},
\end{align}
which is non-positive if $\alpha \ge 0$ and $\mat{S}_{J}^{g}$ is positive. Thus
energy is non-increasing across mortars between elements.

Since energy is non-increasing across all mortar elements $g$, we have that
$\frac{d}{dt}\En \le 0$ and the result $\En(t) \le
\En(0)$ for $t > 0$ follows upon integration.

\section{Proof of Theorem~\ref{thm:GD:div}}\label{app:thm:GD:div}
Using~\eref{eqn:Jrx} it follows that
\begin{align}
  \vec{1}^{T} \mat{L}^{T} \mat{W} \mat{J}\mat{s}_{y} \mat{D}_{s} \mat{v}_{y}
  =
  \vec{x}_{r}^{T} \mat{L}^{T} \mat{W} \mat{D}_{s} \mat{v}_{y}
  =
  \int_{\hat{e}} x_{r} \frac{\partial v_{y}}{\partial s}
  =
  \int_{\hat{e}} \frac{\partial x}{\partial r} \frac{\partial v_{y}}{\partial s}
  =
  \int_{\hat{e}} J\frac{\partial s}{\partial y} \frac{\partial v_{y}}{\partial s},
\end{align}
where the second equality follows from the fact that the quadrature is exact
for inner products and the third equality from the
$L^2$-projection~\eref{eqn:xr:int}. Similar calculations for the other metric
terms yield:
\begin{align}
  \vec{1}^{T} \mat{L}^{T} \mat{W} \mat{J}\mat{r}_{y} \mat{D}_{r} \mat{v}_{y}
  &=
  \int_{\hat{e}} J\frac{\partial r}{\partial y} \frac{\partial v_{y}}{\partial r},\\
  \vec{1}^{T} \mat{L}^{T} \mat{W} \mat{J}\mat{r}_{x} \mat{D}_{r} \mat{v}_{x}
  &=
  \int_{\hat{e}} J\frac{\partial r}{\partial x} \frac{\partial v_{x}}{\partial r},\\
  \vec{1}^{T} \mat{L}^{T} \mat{W} \mat{J}\mat{s}_{x} \mat{D}_{s} \mat{v}_{x}
  &=
  \int_{\hat{e}} J\frac{\partial s}{\partial x} \frac{\partial v_{x}}{\partial s}.
\end{align}
Using this with the definitions of $\mat{S}_{x}$ and $\mat{S}_{y}$ gives
\begin{align}
  \vec{1}^{T} \mat{S}_{x} \vec{v}_{x}
  &=
  \int_{\hat{e}} J\left(\frac{\partial r}{\partial x} \frac{\partial v_{x}}{\partial r}
  + \frac{\partial s}{\partial x} \frac{\partial v_{x}}{\partial s}\right)
  =
  \int_{\hat{e}} J \frac{\partial v_{x}}{\partial x},\\
  \vec{1}^{T} \mat{S}_{y} \vec{v}_{y}
  &=
  \int_{\hat{e}} J\left(\frac{\partial r}{\partial y} \frac{\partial v_{y}}{\partial r}
  + \frac{\partial s}{\partial y} \frac{\partial v_{y}}{\partial s}\right)
  =
  \int_{\hat{e}} J \frac{\partial v_{y}}{\partial y}.
\end{align}
Putting these together we then have that
\begin{align}
  \vec{1}^{T} \mat{S}_{x} \vec{v}_{x} + \vec{1}^{T} \mat{S}_{y} \vec{v}_{y}
  =
  \int_{\hat{e}} J \left(\frac{\partial v_{x}}{\partial x} + \frac{\partial v_{y}}{\partial y}\right)
  =
  \int_{\hat{\partial e}} S_{J} \left(n_{x} v_{x} + n_{y} v_{y}\right).
\end{align}

\section{Proof of Theorem~\ref{thm:conserved:constant}}\label{app:thm:conserved:constant}
\subsection{Constant Preservation}
In order to show that constants are preserved, we let the solution be
constant, i.e., for all $e\in\EE$, $p = \beta$, $v_x = \gamma_{x}$, and $v_x =
\gamma_{y}$ with $\beta, \gamma_{x}, \gamma_{y} \in \mathbb{R}$. Since the
fields are the same constant across all mortar elements $g$, we have for
each $g \in \GG^{e}$ that $v_{n}^{-} =
v_{n}^{+} = n_{x}^{-} \gamma_{x} + n_{y}^{-} \gamma_{y} = \gamma_{n}^{-}$ and
$p^{-} = p^{+} = \beta$ and thus $v_{n}^{*} = \gamma_{n}^{-}$ and $p^{*} = \beta$.
This, along with the fact that the $\mat{S}_{x} \vec{1} = \mat{S}_{y} \vec{1} =
\vec{0}$ implies that~\eref{eqn:semi:p}
\begin{align}
  \mat{M}_{J} \frac{d\vec{p}}{dt} = \vec{0}.
\end{align}
Considering now~\eref{eqn:semi:vx} we have that
\begin{align}
  &\mat{M}_{J} \frac{d\vec{v_{x} }}{dt} =
  \mat{S}_{x}^{T} \beta\vec{1}
  -\sum_{g\in\GG^{e}} \beta {(\mat{L}^{g})}^{T} \mat{n}_{x}^{-g} \mat{S}_{J}^{g} \mat{W}^{g}
  \mat{L}^{g}\vec{1},
\end{align}
In order to show that the right-hand side it zero we let $w \in V^{e}_{h}$ and
multiply by $\vec{w}^{T}$:
\begin{align}
  \vec{w}^{T} \mat{S}_{x}^{T} \beta\vec{1}
  -\sum_{g\in\GG^{e}} \beta \vec{w}^{T}{(\mat{L}^{g})}^{T} \mat{n}_{x}^{-g} \mat{S}_{J}^{g} \mat{W}^{g}
  \vec{1}^{g}
  =
  \vec{w}^{T} \mat{S}_{x}^{T} \beta\vec{1}
  -\sum_{g\in\GG^{e}} \beta
  {\left(\mat{n}_{x}^{-g}\vec{w}^{-}\right)}^{T}
  \mat{S}_{J}^{g} \mat{W}^{g} \vec{1}^{g}.
\end{align}
By Theorem~\ref{thm:GD:div} it follows that
\begin{align}
  \vec{w}^{T} \mat{S}_{x}^{T} \beta\vec{1} = \int_{\hat{\partial e}} S_{J} n_{x} w.
\end{align}
and using this along with~\eref{eqn:surf:mort:int} we have that
\begin{align}
  \vec{w}^{T} \mat{S}_{x}^{T} \beta\vec{1}
  -\sum_{g\in\GG^{e}} \beta
  {\left(\mat{n}_{x}^{-g}\vec{w}^{-}\right)}^{T}
  \mat{S}_{J}^{g} \mat{W}^{g} \vec{1}^{g}
  = 0.
\end{align}
Since this must hold for all $w \in V^{e}_{h}$ it follows that
\begin{align}
  &\mat{M}_{J} \frac{d\vec{v_{x} }}{dt} =
  \mat{S}_{x}^{T} \beta\vec{1}
  -\sum_{g\in\GG^{e}} \beta {(\mat{L}^{g})}^{T} \mat{n}_{x}^{-g} \mat{S}_{J}^{g} \mat{W}^{g}
  \mat{L}^{g}\vec{1} = \vec{0}.
\end{align}
A similar calculation shows that
\begin{align}
  &\mat{M}_{J} \frac{d\vec{v_{y} }}{dt} = \vec{0}
\end{align}
and the scheme is constant preserving.

\subsection{Conservation}
In order to show that the scheme is conservative we
multiply~\eref{eqn:semi:p}--\eref{eqn:semi:vy} each by $\vec{1}^{T}$ and sum
over all elements. If the scheme is conservative, then the sum of each
component should be zero. Considering first~\eref{eqn:semi:p} we have that
\begin{align}
  &\vec{1}^{T}\mat{M}_{J} \frac{d\vec{p}}{dt}
  + \vec{1}^{T}\mat{S}_{x} \vec{v}_{x} + \vec{1}^{T}\mat{S}_{x} \vec{v}_{y}
  =
  -\sum_{g\in\GG^{e}} \vec{1}^{T} {(\mat{L}^{g})}^{T} \mat{S}_{J}^{g} \mat{W}^{g}
  \left(\vec{v}^{*}_{n} - \vec{v}^{-}_{n}\right).
\end{align}
Direct application of Theorem~\ref{thm:GD:div} and~\eref{eqn:surf:mort:int}
then gives
\begin{align}
  &\vec{1}^{T}\mat{M}_{J} \frac{d\vec{p}}{dt}
  =
  -
  \sum_{g\in\GG^{e}} \vec{1}^{T} {(\mat{L}^{g})}^{T} \mat{S}_{J}^{g} \mat{W}^{g}
  \vec{v}^{*}_{n}.
\end{align}
Since, by construction, $v_{n}^{*}$ will be equal and opposite on the two
sides of each mortar it follow that after summing over the whole mesh that
\begin{align}
  &\sum_{e\in\EE}\vec{1}^{T}\mat{M}_{J} \frac{d\vec{p}}{dt} = 0,
\end{align}
and the pressure field $p$ is conserved. Similarly,
multiplying~\eref{eqn:semi:vx} and~\eref{eqn:semi:vy} by $\vec{1}^{T}$ gives
\begin{align}
  &\vec{1}^{T}\mat{M}_{J} \frac{d\vec{v_{x} }}{dt}
  = -\sum_{g\in\GG^{e}} \vec{1}^{T}{(\mat{L}^{g})}^{T} \mat{n}_{x}^{-g} \mat{S}_{J}^{g} \mat{W}^{g}
  \vec{p}^{*},\\
  &\vec{1}^{T}\mat{M}_{J} \frac{d\vec{v_{y} }}{dt}
  = -\sum_{g\in\GG^{e}} \vec{1}^{T} {(\mat{L}^{g})}^{T} \mat{n}_{y}^{-g} \mat{S}_{J}^{g} \mat{W}^{g}
  \vec{p}^{*},
\end{align}
where we have used that $\mat{S}_{x}\vec{1} = \mat{S}_{y}\vec{1} = \vec{0}$.
Summing over the whole mesh and using that the normal vectors are equal and
opposite on both sides of the mortar elements gives:
\begin{align}
  &\sum_{e\in\EE}\vec{1}^{T}\mat{M}_{J} \frac{d\vec{v_{x} }}{dt}
  = 0,\\
  &\sum_{e\in\EE}\vec{1}^{T}\mat{M}_{J} \frac{d\vec{v_{y} }}{dt}
  = 0,
\end{align}
thus the scheme conserves both the velocity components. Note that conservation of
velocity does not require the use of Theorem~\ref{thm:GD:div}
and~\eref{eqn:surf:mort:int}, due to the use of the weak derivative.

\bibliographystyle{elsarticle-num}

\end{document}